\documentclass[journal]{IEEEtran}
\usepackage[numbers,sort&compress]{natbib}

\usepackage{algorithm}
\usepackage{algpseudocode}

\usepackage{pgf}
\usepackage{tikz}
\usetikzlibrary{arrows, decorations.pathmorphing, backgrounds, positioning, fit, petri, automata}
\definecolor{yellow1}{rgb}{1,0.8,0.2}

\usepackage{makecell}

\usepackage{authblk}
\usepackage[T1]{fontenc}
\usepackage{inputenc}

\usepackage{amsmath,amsthm,amsfonts,amssymb,amsbsy,amsopn,amstext}
\usepackage{graphicx,color}
\usepackage{mathbbol,mathrsfs}
\usepackage{float,ccaption}
\usepackage{indentfirst}
\usepackage{subfigure}

\usepackage{cuted}

\usepackage{booktabs}
\usepackage{threeparttable}
\usepackage{multirow}
\usepackage{diagbox}

\usepackage{epstopdf}

\ifCLASSINFOpdf
\else
\fi
\newtheorem{thm}{Theorem}

\newtheorem{lem}{Lemma}

\newtheorem{defn}{Definition}

\newtheorem{ass}{Assumption}

\begin{document}
%

\title{Distributed TD Tracking with Linear Function Approximation over Directed Communication Networks}
%
%
%

\author{Haocheng Yang${}^\dagger$,~\IEEEmembership{}\thanks{}
	Shengchao Zhao${}^\dagger$,~\IEEEmembership{}\thanks{${}^\dagger$Haocheng Yang and Shengchao Zhao contributed equally to this work.\\ ${~~~}^*$Corresponding author.} 
	Yongchao Liu${}^*$
	\thanks{Haocheng Yang is with the School of Mathematical Sciences, Dalian University of Technology, Dalian 116024, China (e-mail: yhc@mail.dlut.edu.cn)

Shengchao Zhao is with the School of Mathematics, China University of Mining and Technology, Xuzhou, 221116, China (e-mail: zhaosc@cumt.edu.cn). 

Yongchao Liu is with the School of Mathematical Sciences, Dalian University of Technology, Dalian 116024, China (e-mail: lyc@dlut.edu.cn). }}
\maketitle

\begin{abstract}
We study the policy evaluation problem in multi-agent reinforcement learning (MARL)  over directed communication networks, where  agents cooperate with each other to explore an unknown environment and accomplish a specific task. We propose a Push-Pull-type distributed algorithm, named PP-DTD, for policy evaluation in MARL within the framework of temporal difference (TD) learning with linear function approximation.  PP-DTD integrates TD learning with the Push-Pull mechanism to accommodate directed communication networks, and further utilize variance reduction techniques to enhance both algorithmic stability and convergence rate. We show that PP-DTD achieves linear convergence to a neighborhood of the optimum under constant step-sizes and a convergence rate of $\mathcal{O}({T^{-1}})$ under decaying step-sizes when the sample is independent and identically distributed  or  Markovian.  
To the best of our knowledge, PP-DTD is the first distributed algorithm for  policy evaluation in MARL over directed graphs that achieves a comparable convergence rate to single-agent TD. The numerical experiments on cooperative navigation tasks demonstrate the robustness and effectiveness of PP-DTD.
\end{abstract}


%
\IEEEpeerreviewmaketitle

\section{Introduction}

\IEEEPARstart{r}{e}inforcement Learning (RL) is a standard paradigm for solving sequential decision-making problems, which  learns an optimal policy to maximize expected cumulative rewards through interactions with an environment. RL has achieved widespread success in  high-dimensional visual control \cite{mnih2015human}, strategic game playing \cite{silver2016mastering}, energy management \cite{sivamayil2023systematic}, etc. The increasing complexity of real-world applications, such as decentralized coordination of cooperative robotics \cite{huttenrauch2019deep}, evaluation of complex strategies in game theory \cite{vinyals2019grandmaster}, and smart factories \cite{bahrpeyma2022review}, has motivated researchers to consider Multi-Agent Reinforcement Learning (MARL). In MARL, multiple agents interact within a shared environment and execute actions simultaneously based on local policies, with global state transitions determined by joint actions.  

A fundamental issue in MARL is to evaluate the agents' policies in terms of their long-term discounted cumulative reward. Temporal-difference (TD) learning \cite{sutton1988learning} provides an efficient and practical framework for solving this problem. 
Among TD-based algorithms for MARL, distributed algorithms that employ TD learning with linear function approximation \cite{sutton1998reinforcement} can effectively handle large state spaces in practice and have been  extensively studied \cite{mathkar2016distributed,stankovi2016multi,doan2019finite,sun2020finite,chen2022multi,wang2020decentralized}.  Mathkar and Borkar \cite{mathkar2016distributed} propose the  distributed gossiping TD(0) algorithm, and provide the almost sure convergence for the proposed algorithm by 
ODE approach.  Stanković and Stanković \cite{stankovi2016multi} propose distributed variants of the TD-learning algorithms GTD2 and TDC \cite{sutton2009fast} over time-varying communication networks, and establish weak convergence of the proposed algorithms under Independent and Identically Distributed (i.i.d.) setting. To reduce sampling variance and communication frequency, Chen et al. \cite{chen2022multi} further develop a mini-batch sampling variant of the distributed TDC algorithm, which achieves a near-optimal sample complexity of $\widetilde{\mathcal{O}}(\varepsilon^{-1})$ for obtaining an $\varepsilon$-accurate solution under the Markovian setting. Doan et al. \cite{doan2019finite} propose a distributed TD(0) algorithm with projection technology over time-varying networks, and provide its convergence rate of $\tilde{\mathcal{O}}(T^{-1})$ under  i.i.d. setting. Subsequently, Sun \cite{sun2020finite} et al. further analyze distributed TD(0) algorithm and prove its asymptotic and nonasymptotic convergence under both i.i.d. and Markovian settings, even without the projection step. To address data heterogeneity, Wang et al. \cite{wang2020decentralized} incorporate Gradient Tracking (GT) into decentralized TD(0) and establish linear convergence to a neighborhood of the
optimum  under both i.i.d. and Markovian settings. Additionally, recent works explore TD($\lambda$) learning with linear function approximation for MARL \cite{stankovic2023distributed,doan2021finite,zhu2024decentralized,wu2021byzantine}.  Doan et al. \cite{doan2021finite} propose a distributed consensus-based variant of the popular TD($\lambda$) algorithm, which achieves linear convergence to a neighborhood of the optimum for constant step-sizes, as well as the convergence rate of $\mathcal{O}(T^{-1})$  for decaying step-sizes under the Markovian setting. Stanković et al. \cite{stankovic2023distributed} develop distributed variants of the  algorithms TD($\lambda$) and ETD($\lambda$), and show that the proposed algorithms converge weakly to the solutions of a derived mean ODE. 
Zhu et al. \cite{zhu2024decentralized} propose  a distributed adaptive TD($\lambda$) algorithm by incorporating AMSGRAD \cite{Sashank2018Adam} into the distributed
variant of TD($\lambda$) learning, which achieves linear convergence to  a neighborhood of the
optimum. To address robustness issues in the presence of malicious agents, Wu et al. \cite{wu2021byzantine} propose a trimmed mean‑based Byzantine‑resilient decentralized TD($\lambda$) algorithm, which converges to a neighborhood of a stationary point.

Indeed, the works mentioned above focus on  policy evaluation in MARL over undirected communication networks.  In a realistic network, especially with mobile agents such as autonomous vehicles, drones, or robots \cite{lin2024finite}, asymmetric transmission capabilities among agents often dictate directed network topologies for MARL. More recently, Lin et al.~\cite{lin2024finite} propose a push-sum-type distributed TD algorithm, Push-SA, over directed communication networks, and establish a finite-time bound that asymptotically converges to zero when the sample is Markovian. We contribute to policy evaluation in MARL over directed communication networks by developing a Push-Pull-type distributed TD algorithm and showing that it may achieve a comparable convergence rate   to single-agent TD \cite{bhandari2018finite} under both i.i.d. and Markovian settings. As far as we are concerned, our main contributions are summarized as follows:
\begin{itemize}
	\item [($i$)] We propose a Push-Pull-type distributed TD algorithm, PP-DTD, for evaluating the policy of MARL over directed communication networks, which integrates TD learning with the Push-Pull mechanism \cite{pu2020push} and a hybrid variance reduction technique \cite{cutkosky2019momentum}. Compared with Push-SA \cite{lin2024finite}, PP-DTD uses a row-stochastic matrix for mixing the parameter vector to be learned, and a column-stochastic matrix for tracking the average gradients, which eliminates the need to know the in-degree information of neighboring agents, thereby offering enhanced flexibility. Moreover, PP-DTD employs the hybrid variance reduction technique to improve both algorithmic stability and convergence rate.
	\item [($ii$)] For both i.i.d. and Markovian settings, PP-DTD achieves linear convergence to a neighborhood of the optimum under constant step-sizes, and a convergence rate of $\mathcal{O}(T^{-1})$ to the optimum under decaying step-sizes. The convergence rate of $\mathcal{O}(T^{-1})$   matches that of the standard TD learning algorithm in single-agent RL \cite{bhandari2018finite}.
	We evaluate the performance of our algorithm through extensive numerical simulations on cooperative navigation tasks under both \text{i.i.d.} and Markovian settings. The empirical results  demonstrate the robustness and effectiveness of PP-DTD.
\end{itemize}

The rest of this paper is organized as follows. Section \ref{sec:problem-alg} introduces the problem formulation of  policy evaluation in MARL and describes the PP-DTD algorithm. Section \ref{sec:main_results} establishes the convergence rate of the proposed algorithms under both i.i.d. and Markovian sampling settings. Finally, Section \ref{sec:numerical} presents numerical results that validate the effectiveness of the proposed algorithm.

$\mathbf{Notation.}$ Throughout this paper, we use the following notation. Denote $\mathbb{R}^d$, $\mathbf{1}$ and $\mathbf{I}$ as the $d$ dimensional Euclidean space, the vector of ones and the identity matrix respectively. 
$B= \text{diag}(\mathbf{b}) \in \mathbb{R}^{l \times l}$ denotes the diagonal matrix whose elements are given by the entries of vector $\mathbf{b}\in \mathbb{R}^{l}$. $\|\cdot\|_2$ denotes the $\ell_2$-norm for vector and matrix; $\|\cdot\|$ represents the Frobenius norm of matrix. For any two positive sequences $\{a_k\}$ and $\{b_k\}$, we say $a_k=\mathcal{O}(b_k)$ if there exists a positive constant $c$ such that $a_k\leq cb_k$. $\mathcal{G}=(\mathcal{V},\mathcal{E})$ denotes a directed communication graph, where $\mathcal{V}=\{1,2,\cdots,n\}$ is the vertex set and $\mathcal{E}\subseteq \mathcal{V}\times \mathcal{V}$ is the set of edges $(i,j)$ such that node $j$ can send information to node $i$. A directed communication graph is said to be strongly connected if there exists a directed path between any two nodes.

\section{Problem formulation and algorithm}\label{sec:problem-alg}
Consider MARL that a group of $n$ agents cooperate to evaluate the value function over a strongly connected directed communication graph $\mathcal{G}$ and each agent $i$ locally follows a stationary policy $\pi^i$. Mathematically, MARL can be modeled as a multi-agent Markov decision process (MDP) with the 6-tuple \cite{wang2020decentralized}:
$$(\mathcal{S}, \{ \mathcal{A}^i \}_{i=1}^n, p, \{ r^i \}_{i=1}^n, \gamma, \mathcal{G}),$$
where 
$\mathcal{S}$ denotes the finite global state space shared by all agents, $\mathcal{A}^i$ is a finite set of actions available to agent $i$, $p:\mathcal{S}\times\mathcal{A}\rightarrow \Delta(\mathcal{S})$ is the state transition model with $ \mathcal{A}:=\mathcal{A}^1\times \mathcal{A}^2\times\cdots\times\mathcal{A}^n$, $r^i:\mathcal{S}\times\mathcal{A}\times\mathcal{S}\rightarrow[0,r_{\max}]$ is the reward function of agent $i$, $\gamma$  is the discount factor.
With global state $s \in \mathcal{S}$, each agent $i \in \mathcal{V}$ selects an action $a^i$ from its private set $\mathcal{A}^i$ according to its policy $\pi^i(\cdot \mid s)$. Based on the joint action $\mathbf{a}:=(a^1,a^2,\cdots,a^n)$, the environment moves to $s'$, and agent $i$ receives reward $r^i(s,\mathbf{a},s')$, where the policy $\pi^i$, action set $\mathcal{A}^i$, and reward $r^i$ are private to agent $i$. Then the value function can be defined as 
\begin{align*}
	V(s) = \mathbb{E} \left[ \frac{1}{n} \sum_{i \in \mathcal{V}} \sum_{t=0}^{\infty} \gamma^t r_{t+1}^i \mid s_0 = s \right].
\end{align*}
Equivalently, the value function can  be reformulated as the Bellman equation
\begin{align}\label{Bel-eq}
	&V(s) \notag\\
	&= \frac{1}{n}\sum_{i\in\mathcal{V}}\sum_{s'\in\mathcal{S}}\sum_{\mathbf{a}\in\mathcal{A}}\pi(\mathbf{a}|s)p(s'|\mathbf{a},s)\left[ r^i(s,\mathbf{a},s')+\gamma V(s') \right],
\end{align}
where $\pi(\mathbf{a}|s):=\prod_{i=1}^n\pi^i(a^i|s)$ is the joint policy.
In practice, computing the exact value function from the above equation is difficult due to the curse of dimensionality. A standard approach is to consider linear function approximation \cite{bhandari2018finite}:
\begin{align*}
	\tilde{V}_{\theta}(s) = \phi^{\top}(s) \theta, \quad \forall s \in \mathcal{S},
\end{align*}
where $\phi(s)\in \mathbb{R}^d$ denotes the feature vector of state $s$ and $\theta \in \mathbb{R}^d$ is a parameter vector. When the state space $\mathcal{S}=\{1,2,\cdots,S\}$ is finite, $\tilde{V}_{\theta}\in\mathbb{R}^S$ can be expressed compactly as
\begin{align*}
	\tilde{V}_{\theta} = \Phi\theta,
\end{align*}
where $\Phi:=\left[\phi(1),\phi(2),\cdots,\phi(S)\right]^{\top}$. 
Note that the Bellman equation may not have solution if we simply replace $V$ with its linear approximation $\tilde{V}_{\theta}$. A classic method is to consider the projected Bellman equation
\begin{equation}\label{projected-BE}
	\Phi\theta=\Pi_\phi\left[\mathbf{r}_\pi+\gamma P_\pi\Phi\theta\right],
\end{equation}
where $\Pi_\phi:=\Phi(\Phi^{\top}\text{diag}(\mathbf{d}_\pi)\Phi)^{-1}\Phi^{\top}\text{diag}(\mathbf{d}_\pi)$, $\mathbf{d}_\pi=\left[d_\pi(1),d_\pi(2),\cdots,d_\pi(S)\right]^\top$ is the unique positive stationary probability distribution of the MDP,\footnote{Throughout the paper, we assume  that $\Phi$ has full column rank and  the finite-state Markov chain induced by the MDP is irreducible and aperiodic under the given policy  to  guarantee the well-definedness of of $\Pi_\phi$ and $\mathbf{d}_\pi$.}
\begin{equation}\label{def:r&P}
	\mathbf{r}_\pi:=[r_\pi(1),r_\pi(2),\cdots,r_\pi(S)]^{\top}, ~P_\pi:=[p_{ss'}]_{S\times S}
\end{equation}
with
\begin{equation}\label{def:r&P-1}
	\begin{aligned}
	&r_\pi(s)=\frac{1}{n}\sum_{i\in\mathcal{V}}\sum_{s'\in\mathcal{S}}\sum_{\mathbf{a}\in\mathcal{A}}\pi(\mathbf{a}|s)p(s'|\mathbf{a},s) r^i(s,\mathbf{a},s'),\\
	&p_{ss'}=\sum_{\mathbf{a}\in\mathcal{A}}\pi(\mathbf{a}|s)p(s'|\mathbf{a},s)	
	\end{aligned}
\end{equation}
for all $1\le s,s'\le S$. 
By \cite[Lemma 6]{Tsitsiklis1997TD},  there exists a unique solution  $\theta^*\in\mathbb{R}^d$ to the projected Bellman equation.

We propose a Push-Pull-type distributed TD algorithm, PP-DTD, for solving the projected Bellman equation (\ref{projected-BE}) over directed networks, which reads as follows.
\begin{algorithm}[H] 
	\caption{Push-Pull based Decentralized TD learning Algorithm (PP-DTD)}\label{alg:PP-DTD}
	\label{1}
	\begin{algorithmic}[1] 
		\Require step-sizes $\alpha_t$, parameters $\beta_t$, weight matrices $\mathbf{W}=[ w_{ij} ]$, $\mathbf{M}=[m_{ij}]$, initial values $\{y_0^i\}_{i \in \mathcal{V}}$, $\{q_{-1}^i\}_{i \in \mathcal{V}}$, $\{\theta^{i}_0\}_{i \in \mathcal{V}}$, $\{q^{i}_0\}_{i \in \mathcal{V}}$, where $\theta^{i}_0=\theta_0$ and $q^{i}_0=q_{-1}^i=0$ for any $i\in\mathcal{V}$.
		\Statex For each agent $i$, compute:
		\For {$t=0,1,2,\cdots$}
		\State{\footnotesize     $\tilde{\theta}_{t+1}^i=       (\theta_t^i + \alpha_t {y}_t^i)$,}
		\State{\footnotesize
			\begin{equation*}
				\theta_{t+1}^i	=\left\{ 
				\begin{aligned}
					\sum_{j=1}^n {w}_{ij}\tilde{\theta}_{t+1}^j~~~~~~~~~~~&\text{when $\xi_t^i:=\{s_t,r_t^i,s_{t+1}'\}$ are \text{i.i.d.},}\footnotemark[2]\\
					\mathbf{\Pi}_{\mathbf{\mathcal{X}}} \left[ \sum_{j=1}^n {w}_{ij} \tilde{\theta}_{t+1}^j \right]~&\text{when $\xi_t^i$ are Markovian}.\footnotemark[2]
				\end{aligned}
				\right.
			\end{equation*}
			
		}  
		\State{\footnotesize Observes $\xi_{t+1}^i=\{s_{t+1},r_{t+1}^i,s_{t+2}'\}$, compute stochastic \textit{semigradients} \cite{sutton1998reinforcement}
			$$g(\theta;\xi^i_{t+1}):=\phi(s_t) \left[  r_t^i + \gamma \phi(s'_{t+1})^\top \theta - \phi(s_t)^\top \theta \right]$$
			at points $\theta_{t+1}^i$ and $\theta_{t}^i$;  updates the estimate of the exact semigradient via}
		{\footnotesize \begin{align*}
				q_{t+1}^i &= (1 - \beta_t) \left( q_{t}^i - g(\theta_t^i; \xi_{t+1}^i) \right) + g(\theta_{t+1}^i; \xi_{t+1}^i).
		\end{align*}}
		\State
		{\footnotesize Updates a semigradient tracking variable
			\begin{align*}
				y_{t+1}^i = \sum_{j=1}^n m_{ij}( y_t^j + q_{t+1}^j - q_t^j).
		\end{align*}}
		\EndFor
	\end{algorithmic}
\end{algorithm}
\footnotetext[2]{``i.i.d.'' means that the tuples $\bigl(s_t, \{r_t^i\}_{i\in\mathcal{V}}, s_{t+1}'\bigr)$ are drawn independently from the stationary distribution of the underlying MDP; ``Markovian'' means that the tuples are collected along a single trajectory of a Markov chain.}

In Algorithm 1, $\mathbf{Step}$ ${\mathbf{2}}$ executes a gradient ascent update  on $\theta_t^i$ along the direction of the global semigradient tracker $y_t^i$ and obtains an intermediate variable $\tilde{\theta}^i_{t+1}$.   
$\mathbf{Step~3}$ updates $\theta_{t+1}^i$ via a neighborhood-weighted sum to guarantee consensus among agents. Moreover, when the samples are Markovian, an additional projection step onto the ball $\mathcal{X}=\{\theta:\|\theta\|\le\mathcal{R}\}$ is incorporated into $\mathbf{Step~3}$ to improve algorithmic stability, where the projection radius $\mathcal{R}$ follows the rule given in  \cite[Lemma 7]{bhandari2018finite}.
$\mathbf{Step}$ ${\mathbf{4}}$ employs the hybrid variance reduction technique \cite{cutkosky2019momentum} to estimate the local exact semigradient, which ensures that the semigradient estimation error vanishes asymptotically. $\mathbf{Step}$ ${\mathbf{5}}$ performs a consensus update for $y_{t+1}^i$ based on the standard dynamic consensus framework \cite{nedic2017achieving}, which can track the current global semigradient information consisting of local semigradient estimates.  $\mathbf{Step}$ ${\mathbf{3}}$ and $\mathbf{Step}$ ${\mathbf{5}}$ only require matrices $\mathbf{W}:=[w_{ij}]$ and $\mathbf{M}:=[m_{ij}]$ to be row stochastic and   column stochastic respectively, which decouples the parameter mixing and semigradient tracking communication steps to adapt to directed communication networks.

To simplify the presentation, we introduce the following notation
\begin{align*}
	&{\Theta}_{t+1} := [\theta_{t+1}^1, \cdots, \theta_{t+1}^n]^\top,\quad  \mathbf{Y}_{t+1} := [y_{t+1}^1, \ldots, y_{t+1}^n]^\top, \\
	&\mathbf{Q}_{t+1} := [q_{t+1}^1, \cdots, q_{t+1}^n]^\top,\\ 
	& \mathbf{G}_{t+1} := [g(\theta_{t+1}^1; \xi_{t+1}^1), \cdots, g(\theta_{t+1}^n; \xi_{t+1}^n)]^\top, \\
	&\mathbf{\tilde{G}}_{t+1} := [g(\theta_t^1; \xi_{t+1}^1), \cdots, g(\theta_t^n; \xi_{t+1}^n)]^\top.
\end{align*}
Consequently, Algorithm \ref{alg:PP-DTD} can be expressed in the following compact form:
\begin{equation}\label{alg-1}
	\begin{aligned}
		\begin{cases}
			{\Theta}_{t+1}&=\mathbf{W}({\Theta}_t+\alpha_t\mathbf{Y}_t) \\
			&\quad\left(\text{or }\mathbf{\Pi}_X\left[\mathbf{W}({\Theta}_t+\alpha_t\mathbf{Y}_t)\right] \text{ for Markovian setting}\right),\\
			\mathbf{Q}_{t+1}&=(1-\beta_t)\mathbf{Q}_t-(1-\beta_t)\widetilde{\mathbf{G}}_{t+1}+\mathbf{G}_{t+1},\\
			\mathbf{Y}_{t+1}&=\mathbf{M}(\mathbf{Y}_t+\mathbf{Q}_{t+1}-\mathbf{Q}_t),
		\end{cases}
	\end{aligned}	
\end{equation}
where  $X:=\underbrace{\mathcal{X}\times\mathcal{X}\times\cdots\times\mathcal{X}}_{n}$.

\section{Convergence analysis}\label{sec:main_results}
In this section, we provide the convergence rates of the PP-DTD algorithm. We first introduce some needed definition, assumptions, and technical lemma that will be used throughout the paper.

\begin{ass}
	\label{ass2}
	All features  $\phi(i)$  are  bounded, i.e. $\|\phi(i)\| \le 1$. 
\end{ass}

\begin{ass}[\textbf{weight matrices and networks}]\label{ass4}
	Let $\mathcal{G}_{\mathbf{W}}=\left(\mathcal{V},\mathcal{E}_\mathbf{W}\right)$ and $\mathcal{G}_{\mathbf{M}^\top}=\left(\mathcal{V},\mathcal{E}_{\mathbf{M}^\top}\right)$ be subgraphs of $\mathcal{G}$ induced by  matrices $\mathbf{W}$ and $\mathbf{M}^\top$ respectively\footnote[3]{For a nonnegative weight matrix $\mathbf{W}=\{w_{ij}\}\in\mathbb{R}^{n\times n}$,  define the induced directed communication graph as $\mathcal{G}_{\mathbf{W}}=(\mathcal{V},\mathcal{E}_{\mathbf{W}})$ where $(i, j)\in \mathcal{E}_{\mathbf{W}}$ if and only if $w_{ij}>0$.}.  Suppose that  
	\begin{itemize}
		\item [(i)]The matrix $\mathbf{W}\in\mathbb{R}^{n\times n}$ is nonnegative row stochastic and $\mathbf{M}\in\mathbb{R}^{n\times n}$ is nonnegative column stochastic, i.e., $\mathbf{W}\mathbf{1}=\mathbf{1}$ and $\mathbf{1}^\intercal\mathbf{M}=\mathbf{1}^\top$. In addition, the diagonal entries of $\mathbf{W}$ and $\mathbf{M}$ are positive.
		\item[(ii)]  The graphs $\mathcal{G}_\mathbf{W}$ and $\mathcal{G}_{\mathbf{M}^\top}$ each contain at least one spanning tree. Moreover, there exists at least one node that is a root of spanning trees for both $\mathcal{G}_\mathbf{W}$ and $\mathcal{G}_{\mathbf{M}^\top}$, i.e. $\mathcal{R}_\mathbf{W}\cap \mathcal{R}_{\mathbf{M}^\top}\ne\emptyset$, where $\mathcal{R}_\mathbf{W}$ ( $\mathcal{R}_{\mathbf{M}^\top}$) is the set of roots of all possible spanning trees in the graph $\mathcal{G}_\mathbf{W}$ ( $\mathcal{G}_{\mathbf{M}^\top}$).
	\end{itemize}
\end{ass}

Assumption~\ref{ass4} is a standard condition of the underlying network for Push-Pull-type algorithms \cite{pu2020push,song2022compressed}, which does not require $\mathcal{G}_{\mathbf{W}}$ and $\mathcal{G}_{\mathbf{M}^\top}$ to be undirected or strongly connected. Under Assumption \ref{ass4}, matrix $\mathbf{W}$ has a nonnegative left eigenvector $\mathbf{u}$ with $\mathbf{u}^{\top}\mathbf{1} = n$, and matrix $\mathbf{M}$ has a nonnegative left eigenvector $\mathbf{v}$ with $\mathbf{v}^{\top}\mathbf{1} = n$; moreover, $\mathbf{u}^{\top}\mathbf{v} > 0$ \cite[Lemma 1]{pu2020push}.

To facilitate the quantification of the consensus error of PP-DTD,  we define the following matrix norm.

\begin{defn}
	\emph{Given an arbitrary inner product $\langle x,y \rangle_W :=\langle\hat{\mathbf{W}}x,\hat{\mathbf{W}}y \rangle$ and its induced vector norm $\|x\|_W:=\|\hat{\mathbf{W}}x\|$ on $\mathbb{R}^{n}$, for any $\mathbf{x},\mathbf{y}\in \mathbb{R}^{n\times d}$,
		\begin{align*}
			&\langle\mathbf{x},\mathbf{y} \rangle_W=\langle\mathbf{x}^{(1)},\mathbf{y}^{(1)} \rangle_W+\langle\mathbf{x}^{(2)},\mathbf{y}^{(2)} \rangle_W+\cdots+\langle\mathbf{x}^{(d)},\mathbf{y}^{(d)} \rangle_W
		\end{align*}
		and $\|\mathbf{x}\|_W=\sqrt{\langle\mathbf{x},\mathbf{x} \rangle_W}$, where  $\hat{\mathbf{W}}\in\mathbb{R}^{n\times n}$ is an invertible matrix, $\mathbf{x}^{(i)}$ and $\mathbf{y}^{(i)}$ are the $i$-th column of matrix $\mathbf{x}$ and $\mathbf{y}$ respectively.}
\end{defn}

\begin{lem}\cite[Lemma 3]{song2022compressed}\label{lem99}
	Under Assumption \ref{ass4},\\ (i) there exist invertible matrices $\hat{\mathbf{W}}$, $\hat{\mathbf{M}}\in\mathbb{R}^{n\times n}$ and the corresponding induced inner products
	\begin{equation*}
		\langle x, y\rangle_{\mathbf{M}}:=\langle \hat{\mathbf{M}}x, \hat{\mathbf{M}}y\rangle,\quad \langle x, y\rangle_{{\mathbf{W}}}:=\langle \hat{\mathbf{W}}x, \hat{\mathbf{W}}y\rangle
	\end{equation*}
	and vector norms
	\begin{equation*}    \|x\|_{\hat{\mathbf{M}}}:=\left\|\hat{\mathbf{M}}x\right\|,\quad \|x\|_{\hat{\mathbf{W}}}:=\left\|\hat{\mathbf{W}}x\right\|,\quad \forall \mathbf{x}\in\mathbb{R}^{n};
	\end{equation*}
	(ii) let $\|\cdot\|_*$ and $\|\cdot\|_{**}$ be any two vector norms of $\|\cdot\|$, $\|\cdot\|_{\hat{{\mathbf{M}}}}$ or $\|\cdot\|_{\hat{{\mathbf{W}}}}$. There exists a constant $\bar{c}>1$ such that
	\begin{equation*}
		\|x\|_*\le \bar{c} \|x\|_{**},\quad \forall x\in\mathbb{R}^{n};
	\end{equation*}
	(iii) the corresponding matrix norms satisfy:
	\begin{equation*}
		\left\|\mathbf{M}-\frac{v\mathbf{1}^\intercal}{n}\right\|_{\hat{\mathbf{M}}}\le1-\rho_\mathbf{M}, \quad \left\|\mathbf{W}-\frac{\mathbf{1}u^\intercal}{n}\right\|_{\hat{\mathbf{W}}}\le 1-\rho_\mathbf{W},
	\end{equation*}
	where $\rho_\mathbf{M},\rho_\mathbf{W}$ are constants in $(0,1]$.
\end{lem}

\subsection{Convergence analysis under \text{i.i.d.} setting}
In this subsection, we present the convergence rate of PP-DTD under \text{i.i.d.} setting, that is,   the tuples $\bigl(s_t, \{r_t^i\}_{i\in\mathcal{V}}, s_{t+1}'\bigr)$ are drawn independently from the stationary distribution of the underlying MDP.

For the sake of analysis, we define  
\begin{equation}\label{ea}
	{\small\left\{
	\begin{aligned}
		({i})&\text{ local exact semigradients and its estimation errors:}\\
		&\mathbf{G}_{t+1}^{\mathbb{E}} := \left[ g^1(\theta_{t+1}^1), \dots, g^n(\theta_{t+1}^n) \right],\\ 
		&\mathbf{e}_{t+1} := \left[ {e}_{t+1}^1, \dots, {e}_{t+1}^n \right]^{\top},\\
		(ii)&\text{ global exact semigradient and averaged iterate: }\\
		& \bar{g}(\theta):=\frac{1}{n}\sum_{j=1}^ng^j(\theta),~~\bar{\theta}_t=\sum_{j=1}^n \frac{u_{j}}{n} {\theta}_{t}^{j},\\
		(iii)&\text{ weighted average matrices:}\\
		&\bar{\mathbf{Y}}_t = \frac{\mathbf{v} \mathbf{1}^\top}{n} \mathbf{Y}_t,\bar{\mathbf{G}}_t = \frac{\mathbf{v} \mathbf{1}^\top}{n} \mathbf{G}_t^{\mathbb{E}},\bar{\Theta}_t=\frac{\mathbf{1}{\mathbf{u}^\top}}{n}{\Theta}_t,
	\end{aligned}
	\right.}
\end{equation}where 
\begin{align}
	g^i(\theta) :=& \sum_{s \in \mathcal{S}} d_\pi(s) r^i_\pi(s)\phi(s) \\
	&+ \sum_{s, s' \in \mathcal{S}} d_\pi(s) p_{ss'} \left( \gamma \phi(s')^\top \theta - \phi(s)^\top \theta \right) \phi(s),\notag\\
	e^i_{t+1}:=&g^i(\theta_{t+1}^i)-q^i_{t+1},\notag\\
	r^i_\pi(s):=&\sum_{s'\in\mathcal{S}}\sum_{\mathbf{a}\in\mathcal{A}}\pi(\mathbf{a}|s)p(s'|\mathbf{a},s) r^i(s,\mathbf{a},s'),\label{def:r-i}
\end{align}
$d_\pi(s)$ and $p_{ss'}$ are defined in (\ref{def:r&P}).
Next, we establish the convergence rate of PP-DTD in three steps:
\begin{itemize}
	\item [($i$)] Establish recursive bounds for the global semigradient estimation error $\mathbb{E} \left[ \| \mathbf{e}_{t+1} \|^2 \right]$, global semigradient tracking error $\mathbb{E}\left[\left\|{\mathbf{Y}}_{t+1} - \bar{\mathbf{Y}}_{t+1}\right\|^2\right]$, and consensus error $\mathbb{E}\left[\left\|{\Theta}_{t+1}-\bar{{\Theta}}_{t+1}\right\|_{\widehat{\mathbf{M}}}^{2}\right]$ (Lemmas \ref{lem2}–\ref{lem4});
	\item[($ii$)] Derive recursive bounds for the optimality gap $\mathbb{E}\left[\left\|\bar{{\theta}}_{t+1}-{\theta}^{*}\right\|_2^{2}\right]$ (Lemma \ref{lem5});
	\item[($iii$)] Construct a combined Lyapunov function and establish its convergence rate (Theorem \ref{t1}).
\end{itemize}

\begin{lem}\label{lem2}
	Suppose that Assumptions \ref{ass2}, \ref{ass4} hold and $\beta_t \leq \frac{1}{2}$, $\alpha_t^2 \leq \frac{\|\mathbf{W}-\mathbf{I}\|^2}{2(1+\gamma)^2}$. Then for any $t \ge 0$,
	\begin{align*} 
		&\mathbb{E} \left[ \| \mathbf{e}_{t+1} \|^2 \right]\\
		\leq &(1 - 2 \beta_t  + \beta_t^2 + 16(1+\gamma)^2\alpha_t^2) \mathbb{E} \left[ \| \mathbf{e}_t \|^2 \right]+ 6 \sigma^2\beta^2_t \notag\\
		&+ 4(1+\gamma)^2(1+8 \| \mathbf{W}-\mathbf{I} \|^2) \bar{c}^2\mathbb{E} \left[ \| \Theta_t - \bar{\Theta}_t \|^2_{\widehat{\mathbf{W}}} \right] \notag\\
		&+ { 4(1+\gamma)^2 n (3\beta_t^2 + 8 (1 + \gamma)^2 n \alpha_t^2 )} \mathbb{E} \left[ \| \bar{\theta}_t - \theta^* \|_2^2 \right]\notag\\
		&+ 16(1+\gamma)^2 \alpha_t^2 \bar{c}^2\mathbb{E} \left[ \| \mathbf{Y}_t - \bar{\mathbf{Y}}_{t} \|^2_{\widehat{\mathbf{M}}} \right],\notag 
	\end{align*}
	where $\sigma^2:=\sum_{j=1}^n\mathbb{E}\left[\left\| g(\theta^*;\xi^j)\right\|^2\right]$, $\bar{c}$ is defined in Lemma \ref{lem99}.
\end{lem}

\begin{proof}
	By the definition of $\mathbf{e}_{t+1}$ in (\ref{ea}),
	{\small\begin{align*}
		&\| \mathbf{e}_{t+1} \|^2\\
		 &= \| (1 - \beta_t) \mathbf{e}_t \|^2 + \| (1 - \beta_t) (\mathbf{G}_{t+1}^\mathbb{E} - \widetilde{\mathbf{G}}_{t+1}) + (\mathbf{G}_{t+1} - \mathbf{G}_{t+1}^\mathbb{E}) \|^2 \notag\\
		&\quad+ 2 \left< (1 - \beta_t) \mathbf{e}_t, (1 - \beta_t) (\mathbf{G}_{t+1}^\mathbb{E} - \widetilde{\mathbf{G}}_{t+1}) + (\mathbf{G}_{t+1} - \mathbf{G}_{t+1}^\mathbb{E}) \right>.\notag
	\end{align*}}
	
	Taking the expectation of both sides of above equation, 
{\small	\begin{align}
		&\mathbb{E} \left[ \| \mathbf{e}_{t+1} \|^2 \right]\notag\\
		 &= (1 - \beta_t)^2 \mathbb{E} \left[ \| \mathbf{e}_t \|^2 \right]\notag\\
		&\quad + \mathbb{E} \left[ \left\| (1 - \beta_t) (\mathbf{G}_{t+1}^\mathbb{E} - \widetilde{\mathbf{G}}_{t+1}) + (\mathbf{G}_{t+1} - \mathbf{G}_{t+1}^\mathbb{E}) \right\|^2 \right] \notag\\
		&\leq (1 - \beta_t)^2 \mathbb{E} \left[\left \| \mathbf{e}_t \right\|^2 \right]\notag\\
		&\quad+2\mathbb{E} \left[ \left\| (1 - \beta_t) (\mathbf{G}_{t}^\mathbb{E} - {\mathbf{G}}^{\mathbb{E}}_{t+1}) + (1-\beta_t) (\mathbf{G}_{t+1} - \widetilde{\mathbf{G}}_{t+1}) \right\|^2 \right]\notag\\
		&\quad + 2 \beta_t^2\mathbb{E} \left[ \| \mathbf{G}_{t+1} - \mathbf{G}_{t+1}^\mathbb{E} \|^2 \right],\label{e2}
	\end{align}}where the equality is obtained by the unbiasedness of the stochastic semigradient estimators
	under \text{i.i.d.} setting, and the inequality is obtained by Young's inequality.

	For the second term on the right-hand side of (\ref{e2}),
	{\small\begin{align*}
		&2 \mathbb{E} \left[ \left\| (1 - \beta_t) (\mathbf{G}_{t}^\mathbb{E} - \mathbf{G}_{t+1}^\mathbb{E}) + (1 - \beta_t) (\mathbf{G}_{t+1} - \widetilde{\mathbf{G}}_{t+1}) \right\|^2 \right] \notag\\
		&= 2 \mathbb{E} \left[ \| (1 - \beta_t) (\mathbf{G}_{t}^\mathbb{E} - \mathbf{G}_{t+1}^\mathbb{E}) \|^2 \right]\notag \\
		&\quad+ 4 \mathbb{E} \left[ (1 - \beta_t)^2 \mathbb{E} \left[ \left< \mathbf{G}_{t}^\mathbb{E} - \mathbf{G}_{t+1}^\mathbb{E}, \mathbf{G}_{t+1} - \widetilde{\mathbf{G}}_{t+1} \right> \mid {\Theta}_{t+1},{\Theta}_{t} \right] \right] \notag\\
		&\quad+ 2 \mathbb{E} \left[ \| (1 - \beta_t) (\mathbf{G}_{t+1} - \widetilde{\mathbf{G}}_{t+1}) \|^2 \right] \notag\\
		&= 2 (1 - \beta_t)^2\mathbb{E} \left[ - \| \mathbf{G}_{t}^\mathbb{E} - \mathbf{G}_{t+1}^\mathbb{E} \|^2 + \| \mathbf{G}_{t+1} - \widetilde{\mathbf{G}}_{t+1} \|^2 \right] \notag\\
		&\leq 2 (1 - \beta_t)^2 \mathbb{E} \left[ \| \mathbf{G}_{t+1} - \widetilde{\mathbf{G}}_{t+1} \|^2 \right] \notag\\
		&\leq 2 (1 + \gamma)^2 \mathbb{E} \left[ \| \Theta_{t+1} - \Theta_{t} \|^2 \right]\notag,
	\end{align*}}where the second equality is obtained by the fact that $\mathbb{E} \left[ \mathbf{G}_{t+1} - \widetilde{\mathbf{G}}_{t+1} \mid {\Theta}_{t+1}, {\Theta}_{t} \right] = \mathbf{G}_{t+1}^\mathbb{E} - \mathbf{G}_t^\mathbb{E}$, and the last inequality is obtained by the $(1+\gamma)$-Lipschitz continuity of $g(\cdot;\xi_{t+1}^i)$.

	For the third term on the right-hand side of (\ref{e2}),
	{\small\begin{align}
		&2 \beta_t^2 \mathbb{E} \left[ \left\| \mathbf{G}_{t+1}^\mathbb{E} - \mathbf{G}_{t+1} \right\|^2 \right]\notag\\ 
		&= 2 \beta_t^2 \left( \mathbb{E} \left[ \left\| \mathbf{G}_{t+1} \right\|^2 \right] - \mathbb{E} \left[ \left\| \mathbf{G}_{t+1}^\mathbb{E} \right\|^2 \right] \right) \notag\\
		&= 2 \beta_t^2  \sum_{i=1}^n \mathbb{E} \left[ \left\| g(\theta_{t+1}^i;\xi^j_{t+1}) - g(\theta^*;\xi^j_{t+1}) + g(\theta^*;\xi^j_{t+1}) \right\|_2^2 \right] \notag\\
		&\leq 6 \beta_t^2 \sigma^2 + 3 \beta_t^2 (1+\gamma)^2 \mathbb{E} \left[ \| \Theta_{t+1} - \Theta^* \|^2 \right] \notag\\
		&\leq 6 \beta_t^2 \sigma^2 + 6 \beta_t^2 (1+\gamma)^2 \mathbb{E} \left[ \left\| \Theta_{t+1} - \Theta_t \right\|^2 \right] \notag\\
		&\quad+ 12 \beta_t^2 (1+\gamma)^2 \mathbb{E} \left[ \left\| \Theta_t - \bar{\Theta}_t \right\|^2 \right]\notag \\
		&\quad+ 12 \beta_t^2 (1+\gamma)^2 n\mathbb{E} \left[ \left\| \bar{\theta}_t - \theta^* \right\|_2^2 \right],\label{e3}
	\end{align}}where  $\sigma^2=\sum_{j=1}^n\mathbb{E}\left[\left\| g(\theta^*;\xi^j)\right\|^2\right]$, and the first equality is obtained by $\mathbb{E}[\mathbf{G}_{t+1} | \mathcal{F}_{t+1}] = \mathbf{G}_{t+1}^{\mathbb{E}}$, the first inequality is obtained by the fact $(a+b)^2 \leq \left(1+\frac{1}{\epsilon}\right) b^2 + (1+\epsilon) a^2$, $\forall a, b, \epsilon > 0$. Then
	{\small\begin{align}
		\mathbb{E}\left[\|\mathbf{e}_{t+1}\|^2\right]
		 &\leq (1-\beta_t)^2 \mathbb{E}\left[\|\mathbf{e}_t\|^2\right] + c_t \mathbb{E}\left[\|{\Theta}_{t+1} - {\Theta}_{t}\|^2\right]\notag\\
		  &\quad + 6\sigma^2 \beta_t^2+12(1+\gamma)^2 \beta_t^2 \mathbb{E}\left[\|\Theta_t - \bar{\Theta}_t\|^2\right] \notag\\
		  &\quad+ 12(1+\gamma)^2n \beta_t^2 \mathbb{E}\left[\|\bar{\theta}_t - \theta^\ast\|_2^2\right], \label{e56}
	\end{align}}
where $c_t:=2(1+\gamma)^2(1+3\beta_t^2)$. 
	
	For the second term on the right-hand side of (\ref{e56}),
	{\small\begin{align*}
		&c_t \mathbb{E}\left[\left\|\Theta_{t+1} - \Theta_{t}\right\|^2\right] \\
		&= {c_t} \mathbb{E}\left[\|(\mathbf{W}-\mathbf{I})\Theta_t - \alpha_t \mathbf{Y}_t\|^2\right]\\
		&= c_t\mathbb{E}\left[\left\|(\mathbf{\mathbf{W}}-\mathbf{\mathbf{I}})(\Theta_t - \bar{\Theta}_t) - \alpha_t (\mathbf{Y}_t - \bar{\mathbf{Y}}_t + \bar{\mathbf{Y}}_t - \bar{\mathbf{G}}_t + \bar{\mathbf{G}}_t)\right\|^2\right]\\
		&\leq 4c_t \left( \|\mathbf{W}-\mathbf{I}\|^2 \mathbb{E}[\|\Theta_t - \bar{\Theta}_t\|^2] + \alpha_t^2 \mathbb{E}\left[\|\mathbf{Y}_t - \bar{\mathbf{Y}}_t\|^2\right] \right.\\
		&\quad\left.+ \alpha_t^2 \mathbb{E}\left[\|\bar{\mathbf{Y}}_t - \bar{\mathbf{G}}_t\|^2\right] + \alpha_t^2 \mathbb{E}\left[\|\bar{\mathbf{G}}_t\|^2\right]\right),
	\end{align*}
	}Notice that $\bar{\mathbf{Y}}_t  = \frac{\mathbf{v} \mathbf{1}^\top}{n} \mathbf{Q}_t$,
	\begin{align*}
		\mathbb{E}\left[\|\bar{\mathbf{Y}}_t - \bar{\mathbf{G}}_t\|^2\right]
		&= \mathbb{E}\left[\left\|\frac{\mathbf{v} \mathbf{1}^\top}{n} (\mathbf{Q}_t - \mathbf{G}_t^{\mathbb{E}})\right\|^2\right]\\
		&\leq \left\|\frac{\mathbf{v} \mathbf{1}^\top}{n}\right\|_2^2 \mathbb{E}\left[\|\mathbf{Q}_t - \mathbf{G}_t^{\mathbb{E}}\|^2\right]\\
		&= \mathbb{E}\left[\|\mathbf{Q}_t - \mathbf{G}_t^{\mathbb{E}}\|^2\right] = \mathbb{E}\left[\|\mathbf{e}_t\|^2\right].
	\end{align*}
	On the other hand,
	{\small\begin{align*}
		\mathbb{E}[\|\bar{\mathbf{G}}_t\|^2] &= \mathbb{E}[\|\bar{\mathbf{G}}_t - \bar{\mathbf{G}}^*\|^2]\\
		 &\leq \left\|\frac{\mathbf{v} \mathbf{1}^\top}{n}\right\|_2^2 \mathbb{E}[\|\mathbf{G}_t^{\mathbb{E}} - \mathbf{G}^\ast\|^2]\\&\leq (1+\gamma)^2 \mathbb{E}[\|\Theta_t - \Theta^\ast\|^2]\\&\leq 2(1+\gamma)^2 \mathbb{E}[\|\Theta_t - \bar{\Theta}_t\|^2] + 2(1+\gamma)^2 n\mathbb{E}[\|\bar{\theta}_t - \theta^\ast\|_2^2],
	\end{align*}
	}where $\mathbf{G}^\ast := \left[ g^1(\theta^\ast), \dots, g^n(\theta^\ast) \right]^\top$, the first equality is obtained by the fact that $\bar{\mathbf{G}}^\ast = \frac{\mathbf{v} \mathbf{1}^\top}{n} \mathbf{G}^\ast = \left[ v_{1}\bar{g}(\theta^*), \dots, v_{n} \bar{g}(\theta^*) \right]^\top = \mathbf{0}$, and the last inequality is obtained by the $(1+\gamma)$-Lipschitz continuity of $g(\cdot;\xi^i_{t+1})$. Then
	\begin{align*}
		&c_t\mathbb{E}\left[\left\|\Theta_{t+1} - \Theta_{t}\right\|^2\right]\\
		&\leq 4c_t \left( \|\mathbf{W}-\mathbf{I}\|^2 \mathbb{E}[\|\Theta_t - \bar{\Theta}_t\|^2] + \alpha_t^2 \mathbb{E}\left[\|\mathbf{Y}_t - \bar{\mathbf{Y}}_t\|^2\right] \right.\\
		&\quad\left.+ \alpha_t^2 \mathbb{E}\left[\|\mathbf{e}_t\|^2\right] + 2\alpha_t^2(1+\gamma)^2 \mathbb{E}[\|\Theta_t - \bar{\Theta}_t\|^2]\right.\\
		&\quad \left.+ 2\alpha_t^2(1+\gamma)^2 n\mathbb{E}[\|\bar{\theta}_t - \theta^\ast\|_2^2] \right).
	\end{align*}
	
	Substituting the above inequality into (\ref{e56}),
	{\small\begin{align}
		&\mathbb{E}\left[\|\mathbf{e}_{t+1}\|^2\right]\notag\\
		 &\leq (1-\beta_t)^2 \mathbb{E}\left[\|\mathbf{e}_t\|^2\right] + \left[12(1+\gamma)^2 \beta_t^2\right.\notag\\
		  &\quad\left.+ 4c_t (\|\mathbf{W}-\mathbf{I}\|^2 + 2(1+\gamma)^2 \alpha_t^2)\right] \mathbb{E}\left[\|\Theta_t - \bar{\Theta}_t\|^2\right] \notag\\
		&\quad + 4c_t \alpha_t^2 \mathbb{E}\left[\|\mathbf{Y}_t - \bar{\mathbf{Y}}_t\|^2\right] + 4c_t \alpha_t^2 \mathbb{E}\left[\|\mathbf{e}_t\|^2\right] + 6\sigma^2 \beta_t^2 \notag \\
		&\quad+ \left[12(1+\gamma)^2 n\beta_t^2 + 8c_t (1+\gamma)^2 n\alpha_t^2\right] \mathbb{E}\left[\|\bar{\theta}_t - \theta^\ast\|^2\right] \notag \\
		&\leq (1-\beta_t)^2 \mathbb{E}\left[\|\mathbf{e}_t\|^2\right] + 4(1+\gamma)^2(1+8\|\mathbf{W}-\mathbf{I}\|^2) \mathbb{E}\left[\|\Theta_t - \bar{\Theta}_t\|^2\right] \notag \\
		&\quad+ 16(1+\gamma)^2 \alpha_t^2 \mathbb{E}\left[\|\mathbf{Y}_t - \bar{\mathbf{Y}}_t\|^2\right] + 16(1+\gamma)^2 \alpha_t^2 \mathbb{E}\left[\|\mathbf{e}_t\|^2\right]\notag \\
		&\quad + 6\sigma^2 \beta_t^2 + 4(1+\gamma)^2n(3\beta_t^2 + 8 (1 + \gamma)^2\alpha_t^2) \mathbb{E}\left[\|\bar{\theta}_t - \theta^\ast\|^2\right]\notag\\
		&\leq (1 - 2 \beta_t  + \beta_t^2 + 16(1+\gamma)^2\alpha_t^2) \mathbb{E} \left[ \| \mathbf{e}_t \|^2 \right] \notag\\
		&\quad+ 4(1+\gamma)^2(1+8 \| \mathbf{W}-\mathbf{I} \|^2) \bar{c}^2\mathbb{E} \left[ \| \Theta_t - \bar{\Theta}_t \|^2_{\widehat{\mathbf{W}}} \right] \notag\\
		&\quad+ 6 \sigma^2\beta^2_t
		+ 16(1+\gamma)^2 \alpha_t^2 \bar{c}^2\mathbb{E} \left[ \| \mathbf{Y}_t - \bar{\mathbf{Y}}_{t} \|^2_{\widehat{\mathbf{M}}} \right]\notag \\
		&\quad+ { 4(1+\gamma)^2 n (3\beta_t^2 + 8 (1 + \gamma)^2 n \alpha_t^2 )} \mathbb{E} \left[ \| \bar{\theta}_t - \theta^* \|_2^2 \right],\notag
	\end{align}}
	where the second inequality is obtained by the fact that $\beta_t^2 \leq \frac{1}{3}$, $\alpha_t^2 \leq \frac{\|\mathbf{W}-\mathbf{I}\|^2}{2(1+\gamma)^2}$ and last inequality is obtained by Lemma \ref{lem99} ($ii$). The proof is complete.
\end{proof}

\begin{lem}\label{lem3}
	Suppose that Assumptions \ref{ass2}, \ref{ass4} hold and $\beta_t\le\frac{1}{2}$, $\alpha_t^2 \leq \frac{\|\mathbf{W}-\mathbf{I}\|^2}{2(1+\gamma)^2}$. Then for any $t \ge 0$,
	{\small\begin{align*}
		&\mathbb{E}\left[\|{\mathbf{Y}}_{t+1} - \bar{\mathbf{Y}}_{t+1}\|^2_{\widehat{\mathbf{M}}}\right]\\
		 &\leq \left(1 - \rho_\mathbf{M} + \frac{24\bar{c}^2}{\rho_\mathbf{M}} {(1+\gamma)^2 \alpha_t^2}\right) \mathbb{E}\left[\|{\mathbf{Y}}_t - \bar{\mathbf{Y}}_t\|^2_{\widehat{\mathbf{M}}}\right]\\
		 &\quad + \frac{3\bar{c}^2}{\rho_\mathbf{M}}(\beta_t^2 + 8\alpha_t^2) \mathbb{E}[\|\mathbf{e}_{t}\|^2] \\
		&\quad+ \frac{48\bar{c}^4}{\rho_\mathbf{M}} \|\mathbf{W}-\mathbf{I}\|_{\widehat{\mathbf{W}}}^2     (1+\gamma)^2 \mathbb{E}\left[\|\Theta_t - \bar{\Theta}_t\|_{\widehat{\mathbf{W}}}^2\right]\\
		&\quad+ \frac{48\bar{c}^2}{\rho_\mathbf{M}} (1+\gamma)^4 n \alpha_t^2 \mathbb{E}[\|\bar{\theta}_t - \theta^\ast\|_2^2] + \frac{9\bar{c}^2}{\rho_\mathbf{M}} \sigma^2 \beta_t^2,   
	\end{align*}
}where $\sigma^2$ is defined in Lemma \ref{lem2}, $\rho_{\mathbf{M}}$ and $\bar{c}$ are defined in Lemma \ref{lem99}.
\end{lem}
\begin{proof}
	By the definition of $\mathbf{Y}_t$ in (\ref{ea}),
	{\small\begin{align}
		\mathbf{Y}_{t+1} - \bar{\mathbf{Y}}_{t+1} = \left( \mathbf{M} - \frac{\mathbf{v}\mathbf{1}^\top}{n} \right) ( (\mathbf{Y}_t - \bar{\mathbf{Y}}_t) + \mathbf{Q}_{t+1} - \mathbf{Q}_t).\label{e4}
	\end{align}
	}Taking $\| \cdot \|^2_{\widehat{\mathbf{M}}}$ on both sides of (\ref{e4}),
	{\small\begin{align}
		&\|\mathbf{Y}_{t+1} - \bar{\mathbf{Y}}_{t+1}\|_{\widehat{\mathbf{M}}}^2 \notag\\
		&\leq (1+\epsilon) \left\|\left(\mathbf{M} - \frac{\mathbf{v}\mathbf{1}^\top}{n}\right) (\mathbf{Y}_t - \bar{\mathbf{Y}}_t)\right\|_{\widehat{\mathbf{M}}}^2 \notag\\
		&\quad+ \left(1+\frac{1}{\epsilon}\right) \left\|\left(\mathbf{M} - \frac{\mathbf{v}\mathbf{1}^\top}{n}\right) (\mathbf{Q}_{t+1} - \mathbf{Q}_t)\right\|_{\widehat{\mathbf{M}}}^2 \notag \\
		&\leq (1+\epsilon)(1-\rho_{\mathbf{M}})^2 \|\mathbf{Y}_t - \bar{\mathbf{Y}}_t\|_{\widehat{\mathbf{M}}}^2 \notag\\
		&\quad+ \left(1+\frac{1}{\epsilon}\right)(1-\rho_{\mathbf{M}})^2 \|\mathbf{Q}_{t+1} - \mathbf{Q}_t\|_{\widehat{\mathbf{M}}}^2 \notag \\
		&\leq (1-\rho_{\mathbf{M}}) \|\mathbf{Y}_t - \bar{\mathbf{Y}}_t\|_{\widehat{\mathbf{M}}}^2 + \frac{1}{\rho_\mathbf{M}} \|\mathbf{Q}_{t+1} - \mathbf{Q}_t\|_{\widehat{\mathbf{M}}}^2, \label{e5}
	\end{align}
	}where the first inequality follows from the fact that $(a+b)^2 \leq \left(1+\frac{1}{\epsilon}\right) b^2 + (1+\epsilon) a^2$, $\forall a, b, \epsilon > 0$, and the last inequality is obtained by setting $\epsilon=\frac{\rho_{\mathbf{M}}}{1-\rho_{\mathbf{M}}}$.
	
	For the second term on the right-hand side of (\ref{e5}),
	\begin{align*}
		&\frac{1}{\rho_{\mathbf{M}}} \|\mathbf{Q}_{t+1} - \mathbf{Q}_t\|_{\widehat{\mathbf{M}}}^2\\ &= \frac{1}{\rho_{\mathbf{M}}} \| -\beta_t \mathbf{Q}_t - (1-\beta_t) \widetilde{\mathbf{G}}_{t+1} + \mathbf{G}_{t+1} \|_{\widehat{\mathbf{M}}}^2 \\
		&\leq \frac{\bar{c}^2}{\rho_{\mathbf{M}}} \| -\beta_t \mathbf{Q}_t - (1-\beta_t) \widetilde{\mathbf{G}}_{t+1} + \mathbf{G}_{t+1} \|^2 \\
		&= \frac{\bar{c}^2}{\rho_{\mathbf{M}}} \| -\beta_t (\mathbf{Q}_t - \mathbf{G}_t^{\mathbb{E}}) - (1-\beta_t) (\widetilde{\mathbf{G}}_{t+1} - \mathbf{G}_{t+1})\\
		&\quad + \beta_t (\mathbf{G}_{t+1}- \mathbf{G}_{t+1}^{\mathbb{E}}) +\beta_t( \mathbf{G}_{t+1}^{\mathbb{E}}-\mathbf{G}_t^{\mathbb{E}}) \|^2\\
		&\leq \frac{3\bar{c}^2}{\rho_{\mathbf{M}}} \left( \beta_t^2 \|\mathbf{e}_t\|^2 + \beta_t^2 \|\mathbf{G}_{t+1} - \mathbf{G}_{t+1}^{\mathbb{E}}\|^2\right.\\
		&\quad \left.+ \|(1-\beta_t) (\widetilde{\mathbf{G}}_{t+1} - \mathbf{G}_{t+1}) + \beta_t (\mathbf{G}_{t+1}^{\mathbb{E}} - \mathbf{G}_t^{\mathbb{E}}) \|^2 \right),
	\end{align*}
	where the first inequality is obtained by Lemma \ref{lem99} ($ii$) and the last inequality is obtained by Young’s inequality. Then
	{\small\begin{align}
		&\mathbb{E} \left[\|\mathbf{Y}_{t+1}-\bar{\mathbf{Y}}_{t+1}\|_{\widehat{\mathbf{M}}}\right]\notag\\
		 &\leq (1-\rho_{\mathbf{M}})\mathbb{E} \left[\|\mathbf{Y}_{t}- \bar{\mathbf{Y}}_{t}\|_{\widehat{\mathbf{M}}}\right]+\frac{3\bar{c}^2}{\rho_{\mathbf{M}}}\beta_t^2\mathbb{E}[\|\mathbf{e}_t\|^2]\notag\\ &\quad+\frac{3\bar{c}^2}{\rho_{\mathbf{M}}}\beta_t^2\mathbb{E}\left[\| \bar{\mathbf{G}}_{t+1}-\mathbf{G}^{\mathbb{E}}_{t+1} \|\right]\notag\\
		 &+\frac{3\bar{c}^2}{\rho_{\mathbf{M}}}\mathbb{E}\left[\| (1-\beta_t)(\widetilde{\mathbf{G}})_{t+1}-\mathbf{G}_{t+1}+\beta_t(\mathbf{G}_{t+1}^{\mathbb{E}}-\mathbf{G}_t^{\mathbb{E}}) \|\right]. \label{e6}
	\end{align}}

	For the third term on the right-hand side of (\ref{e6}), (\ref{e3}) implies that
	{\small\begin{align}
		&\frac{3 \bar{c}^{2}}{\rho_{\mathbf{M}}} \beta_{t}^{2} \mathbb{E}\left[\left\|\mathbf{G}_{t+1}-\mathbf{G}^{\mathbb{E}}_{t+1}\right\|^{2}\right]\notag\\
		 &\leq\frac{9 \bar{c}^{2}}{\rho_{\mathbf{M}}}(1+\gamma) \beta_{t}^{2} \mathbb{E}\left[\left\|{\Theta}_{t+1}-{\Theta}_{t}\right\|^{2}\right]\notag\\
		 &\quad+\frac{18 \bar{c}^{2}}{\rho_{\mathbf{M}}}(1+\gamma) ^{2} \beta_{t}^{2} \mathbb{E}\left[\left\|{\Theta}_{t}-\bar{{\Theta}}_{t}\right\|^{2}\right] \notag\notag
		\\&\quad+ \frac{18 \bar{c}^{2}}{\rho_{\mathbf{M}}}(1+\gamma)^2n \beta_{t}^{2} \mathbb{E}\left[\left\|\bar{{\theta}}_{t}-{\theta}^{*}\right\|_2^{2}\right]+\frac{9 \bar{c}^{2}}{\rho_{\mathbf{M}}} \sigma^{2} \beta_{t}^{2}.\notag
	\end{align}}
	
	For the last term on the right-hand side of (\ref{e6}),
	{\small\begin{align}
		&\frac{3 \bar{c}^{2}}{\rho_{\mathbf{M}}} \mathbb{E}\left[\left\|(1-\beta_{t})\left(\widetilde{\mathbf{G}}_{t+1}-\mathbf{G}_{t+1}\right)+\beta_{t}\left(\mathbf{G}^{\mathbb{E}}_{t+1}-\mathbf{G}^{\mathbb{E}}_{t}\right)\right\|^{2}\right] \notag\\
		&=\frac{3 \bar{c}^{2}}{\rho_{\mathbf{M}}}(1-\beta_{t})^{2} \mathbb{E}\left[\left\|\widetilde{\mathbf{G}}_{t+1}-\mathbf{G}_{t+1}\right\|^{2}\right]\notag\\
		&\quad - \left[(1-\beta_{t})\beta_{t}-\beta_{t}^{2}\right] \mathbb{E}\left[\left\|\mathbf{G}^{\mathbb{E}}_{t+1}-\mathbf{G}^{\mathbb{E}}_{t}\right\|^{2}\right] \notag\\
		&\le \frac{3 \bar{c}^2}{\rho_{\mathbf{M}}}(1-\beta_{t})^{2} \mathbb{E}\left[\left\|\widetilde{\mathbf{G}}_{t+1}-\mathbf{G}_{t+1}\right\|^{2}\right] \notag\\
		&\le \frac{3 \bar{c}^{2}}{\rho_{\mathbf{M}}}(1+\gamma)^{2} \mathbb{E}\left[\left\|{\Theta}_{t+1}-{\Theta}_{t}\right\|^{2}\right],\notag
	\end{align}}where the equality is obtained by the unbiasedness of the stochastic semigradient, the first inequality is obtained by the fact that $\beta_t\le\frac{1}{2}$ and the last inequality is obtained by the $(1+\gamma)$-Lipschitz continuity of $g(\cdot;\xi^i_{t+1})$. Then
	{\small\begin{align*}
		&\mathbb{E}\left[\left\|\mathbf{Y}_{t+1}-\bar{\mathbf{Y}}_{t+1}\right\|_{\widehat{\mathbf{M}}}^{2}\right] \\
		&\leq (1-\rho_{\mathbf{M}}) \mathbb{E}\left[\left\|\mathbf{Y}_{t}-\bar{\mathbf{Y}}_{t}\right\|_{\widehat{\mathbf{M}}}^{2}\right]+ \frac{9 \bar{c^{2}}}{\rho_{\mathbf{M}}} \sigma^{2} \beta_{t}^{2}\\
		&\quad+ \frac{3 \bar{c}^{2}}{\rho_{\mathbf{M}}} \beta_{t}^{2} \mathbb{E}\left[\left\|\mathbf{e}_{t}\right\|^{2}\right] + \frac{3 \bar{c}^{2}}{\rho_{\mathbf{M}}}(1+\gamma)^2 (1+3\beta_{t}^{2}) \mathbb{E}\left[\left\|{\Theta}_{t+1}-{\Theta}_{t}\right\|^{2}\right] \\
		&\quad +\frac{18 \bar{c}^{2}}{\rho_{\mathbf{M}}}(1+\gamma)^2  \beta_{t}^{2} \mathbb{E}\left[\left\|{\Theta}_{t}-\bar{{\Theta}}_{t}\right\|^{2}\right] \notag\\
		&\quad+ \frac{18 \bar{c}^{2}}{\rho_{\mathbf{M}}}(1+\gamma) n \beta_{t}^{2} \mathbb{E}\left[\left\|\bar{{\theta}}_{t}-{\theta}^*\right\|_2^{2}\right] \\
		&\leq (1-\rho_{\mathbf{M}}) \mathbb{E}\left[\left\|\mathbf{Y}_{t}-\bar{\mathbf{Y}}_{t}\right\|_{\widehat{\mathbf{M}}}^{2}\right] + \frac{3 \bar{c}^{2}}{\rho_{\mathbf{M}}}(\beta_{t}^{2} + 8 \alpha_{t}^{2}) \mathbb{E}\left[\left\|\mathbf{e}_{t}\right\|^{2}\right] \\
		&\quad + \frac{24 \bar{c}^{2}}{\rho_{\mathbf{M}}}(1+\gamma)^{2} \left(\|\mathbf{W}-\mathbf{I}\|^{2} + 2 (1+\gamma)^2\alpha^2\right) \mathbb{E}\left[\left\|{\Theta}_{t}-\bar{{\Theta}}_{t}\right\|^{2}\right]  \\
		&\quad + \frac{9 \bar{c}^{2}}{\rho_{\mathbf{M}}} \sigma^{2} \beta_{t}^{2}+ \frac{24 \bar{c}^{2}}{\rho_{\mathbf{M}}}(1+\gamma)^{2} \alpha_{t}^{2} \mathbb{E}\left[\left\|\mathbf{Y}_{t}-\bar{\mathbf{Y}}_{t}\right\|^{2}\right] \\
		&\quad+ \frac{48 \bar{c}^{2}}{\rho_{\mathbf{M}}}(1+\gamma)^4 n  \alpha_{t}^{2} \mathbb{E}\left[\left\|\bar{{\theta}}_{t}-{\theta}^{*}\right\|_2^{2}\right],\\
		&\leq \left(1 - \rho_\mathbf{M} + \frac{24\bar{c}^2}{\rho_\mathbf{M}} {(1+\gamma)^2 \alpha_t^2}\right) \mathbb{E}[\|{\mathbf{Y}}_t - \bar{\mathbf{Y}}_t\|^2]\\
		&\quad + \frac{3\bar{c}^2}{\rho_\mathbf{M}}(\beta_t^2 + 8\alpha_t^2) \mathbb{E}[\|\mathbf{e}_{t}\|^2] \\
		&\quad+ \frac{48\bar{c}^2}{\rho_\mathbf{M}} \|\mathbf{W}-\mathbf{I}\|^2     (1+\gamma)^2 \mathbb{E}\left[\|\Theta_t - \bar{\Theta}_t\|_{\widehat{\mathbf{W}}}^2\right]\\
		&\quad+ \frac{48\bar{c}^4}{\rho_\mathbf{M}} (1+\gamma)^4 n \alpha_t^2 \mathbb{E}[\|\bar{\theta}_t - \theta^\ast\|_2^2] + \frac{9\bar{c}^2}{\rho_\mathbf{M}} \sigma^2 \beta_t^2,
	\end{align*}
	}where the second inequality is obtained by the fact that $\beta_t^2\le\frac{1}{3}$, the last inequality is obtained by the fact that $\alpha^2_t\leq\frac{\|\mathbf{W}-\mathbf{I}\|^2}{2(1+\gamma)^2}$ and lemma \ref{lem99} ($ii$). The proof is complete.
\end{proof}
\begin{lem}\label{lem4}
	Suppose that Assumptions \ref{ass2}, \ref{ass4} hold. Then for any $t \ge 0$,
	\begin{align*}
		&\mathbb{E}\left[\left\|{\Theta}_{t+1}-\bar{{\Theta}}_{t+1}\right\|_{\widehat{\mathbf{M}}}^{2}\right] \\
		&\leq \left(1-\rho_{\mathbf{W}}+\frac{2 \bar{c}^{4}}{\rho_{\mathbf{W}}}(1+\gamma)^{2} \alpha_{t}^{2}\right) \mathbb{E}\left[\left\|{{\Theta}}_{t}-\bar{{{\Theta}}}_{t}\right\|_{\widehat{\mathbf{M}}}^{2}\right]\\
		&\quad + \frac{3 \bar{c}^{4}}{\rho_{\mathbf{W}}} \alpha_{t}^{2} \mathbb{E}\left[\left\|\mathbf{Y}_{t}-\bar{\mathbf{Y}}_{t}\right\|_{\widehat{\mathbf{M}}}^{2}\right] \\
		&\quad+ \frac{3 \bar{c}^{2}}{\rho_{\mathbf{M}}} \alpha_{t}^{2} \mathbb{E}\left[\left\|\mathbf{e}_{t}\right\|^{2}\right] + \frac{2 \bar{c}^{2}}{\rho_{\mathbf{W}}}(1+\gamma)^{2} \alpha_{t}^{2} \mathbb{E}\left[\left\|\bar{{\theta}}_{t}-{\theta}^{*}\right\|_2^{2}\right],
	\end{align*}
	where $\sigma^2$ is defined in Lemma \ref{lem2},  $\bar{c}$, $\rho_{\mathbf{W}}$ and $\rho_{\mathbf{M}}$ are defined in Lemma \ref{lem99}.
\end{lem}
\begin{proof}
	By the iterative formula of $\Theta_t$ and the definition of $\bar{\Theta}_{t}$ in (\ref{ea}),
	\begin{align*}
	{\Theta}_{t+1}-\bar{{\Theta}}_{t+1}& = \left(\mathbf{W}-\frac{\mathbf{1} \mathbf{u}^{\top}}{n} \right)\left({\Theta}_{t}-\alpha_{t} \mathbf{Y}_{t}\right)\\
	& = \left(\mathbf{W}-\frac{\mathbf{1} \mathbf{u}^{\top}}{n}\right)\left(\left({\Theta}_{t}-\bar{{\Theta}}_{t}\right)-\alpha_{t} \mathbf{Y}_{t}\right).
	\end{align*}
	By a similar analysis of (\ref{e5}),
	$$\left\|{\Theta}_{t}-\bar{{\Theta}}_{t+1}\right\|_{\bar{\mathbf{W}}}^{2} \leq \left(1-\rho_{\mathbf{W}}\right)\left\|{\Theta}_{t}-\bar{{\Theta}}_{t}\right\|_{\widehat{\mathbf{W}}}^{2} + \frac{1}{\rho_{\mathbf{W}}}\left\|{\alpha}_{t} \mathbf{Y}_{t}\right\|_{\widehat{\mathbf{W}}}^{2}.$$
	Taking the expectation of both sides of the above inequality,
	\begin{align*}
		&\mathbb{E}\left[\left\|{\Theta}_{t+1}-\bar{{\Theta}}_{t+1}\right\|_{\bar{\mathbf{W}}}^{2}\right]\\
		 &\leq (1-\rho_{\mathbf{W}}) \mathbb{E}\left[\left\|{\Theta}_{t}-\bar{{\Theta}}_{t}\right\|_{\widehat{\mathbf{W}}}^{2}\right] + \frac{\bar{{c}}^{2}}{\rho_{\mathbf{W}}} \alpha_{t}^{2} \mathbb{E}\left[\left\|\mathbf{Y}_{t}\right\|^{2}\right] \\
		&= (1-\rho_{\mathbf{W}}) \mathbb{E}\left[\left\|{\Theta}_{t}-\bar{{\Theta}}_{t}\right\|_{\widehat{\mathbf{W}}}^{2}\right]\\
		&\quad + \frac{\bar{{c}}^{2}}{\rho_{\mathbf{W}}} \alpha_{t}^{2} \mathbb{E}\left[\left\|\mathbf{Y}_{t}-\bar{\mathbf{Y}}_{t} + \bar{\mathbf{Y}}_{t}-\bar{\mathbf{G}}_{t}+\bar{\mathbf{G}}_{t}\right\|^{2}\right] \\
		&\leq (1-\rho_{\mathbf{W}}) \mathbb{E}\left[\left\|{\Theta}_{t}-\bar{{\Theta}}_{t}\right\|_{\widehat{\mathbf{W}}}^{2}\right] + \frac{3 \bar{{c}}^{4}}{\rho_{\mathbf{W}}} \alpha_{t}^{2} \mathbb{E}\left[\left\|\mathbf{Y}_{t}-\bar{\mathbf{Y}}_{t}\right\|^{2}_{\widehat{\mathbf{M}}}\right] \\
		&\quad + \frac{3 \bar{{c}}^{2}}{\rho_{\mathbf{M}}} \alpha_{t}^{2} \mathbb{E}\left[\left\|\mathbf{e}_{t}\right\|^{2}\right] + \frac{2 \bar{{c}}^{4}}{\rho_{\mathbf{W}}}(1+\gamma)^{2} \alpha_{t}^{2} \mathbb{E}\left[\left\|{\Theta}_{t}-\bar{{\Theta}}_{t}\right\|_{\widehat{\mathbf{W}}}^{2}\right] \\
		&\quad + \frac{2 \bar{{c}}^{2}}{\rho_{\mathbf{W}}}(1+\gamma)^{2} n \alpha_{t}^{2} \mathbb{E}\left[\left\|\bar{{\theta}}_{t}-{\theta}^{*}\right\|_2^{2}\right] \\
		&= \left(1-\rho_{\mathbf{W}} + \frac{2 \bar{{c}}^{4}}{\rho_{\mathbf{W}}}(1+\gamma)^{2} \alpha_{t}^{2}\right) \mathbb{E}\left[\left\|{\Theta}_{t}-\bar{{\Theta}}_{t}\right\|_{\widehat{\mathbf{W}}}^{2}\right] \\
		&\quad+ \frac{3 \bar{{c}}^{4}}{\rho_{\mathbf{W}}} \alpha_{t}^{2} \mathbb{E}\left[\left\|\mathbf{Y}_{t}-\bar{\mathbf{Y}}_{t}\right\|_{\widehat{\mathbf{M}}}^{2}\right] \\
		&\quad + \frac{3 \bar{{c}}^{2}}{\rho_{\mathbf{M}}} \alpha_{t}^{2} \mathbb{E}\left[\left\|\mathbf{e}_{t}\right\|^{2}\right] + \frac{2 \bar{{c}}^{2}}{\rho_{\mathbf{W}}}(1+\gamma)^{2} n \alpha_{t}^{2} \mathbb{E}\left[\left\|\bar{{\theta}}_{t}-{\theta}^{*}\right\|_2^{2}\right],
	\end{align*}
	where the second inequality is obtained by Young's inequality and Lemma \ref{lem99} ($ii$). The proof is complete.
\end{proof}

\begin{lem}\label{lem5}
	Suppose that Assumptions \ref{ass2}, \ref{ass4} hold and $\alpha_{t} \le \frac{(1-\gamma)n}{4\mathbf{u}^\top\mathbf{v}}$. Then for any $t \ge 0$,
	\begin{align}
		&\mathbb{E}\left[\left\|\bar{{\theta}}_{t+1}-{\theta}^{*}\right\|_2^{2}\right] \notag\\
		&\leq \left(1-\frac{(1-\gamma)w\mathbf{u}^\top\mathbf{v}}{2n}  \alpha_{t}\right) \mathbb{E}\left[\left\|\bar{{\theta}}_{t}-{\theta}^{*}\right\|_2^{2}\right]\notag\\
		&\quad + \frac{6n}{(1-\gamma) w \mathbf{u}^\top\mathbf{v}} \alpha_{t} \mathbb{E}\left[\left(\frac{\mathbf{u}^\top\mathbf{v}}{n}\right)^{2} \frac{\bar{{c}}^{2}}{n} \left\|{\Theta}_{t}-\bar{{\Theta}}_{t}\right\|_{\widehat{\mathbf{W}}}^{2} \right. \notag\\
		&\left.\quad + \left(\frac{\mathbf{u}^\top\mathbf{v}}{n}\right)^{2} \frac{1}{n} \left\|\mathbf{e}_{t}\right\|^{2} + \bar{{c}}^{2} \left\|\mathbf{Y}_{t}-\bar{\mathbf{Y}}_{t}\right\|_{\widehat{\mathbf{M}}}^{2}\right],\label{e7}
	\end{align}
	where $\sigma^2$ is defined in Lemma \ref{lem2},  $\bar{c}$, $\rho_{\mathbf{W}}$ and $\rho_{\mathbf{M}}$ are defined in Lemma \ref{lem99}, $\omega$ is the minimum eigenvalue
	of matrix $\Phi^{\top}\text{diag}(\mathbf{d}_\pi)\Phi$.
\end{lem}
\begin{proof}
	By the definition of $\bar{\theta}_t$ and the recursion of $\theta_t^i$,
	$$\bar{\theta}_{t+1} = \bar{\theta}_t +\alpha_t \sum_{j=1}^{n} \frac{u_j}{n} y_t^j.$$
	
	Taking the squared norm of the above equation, 
	{\small\begin{align}
		&\left\|\bar{{\theta}}_{t+1}-{\theta}^{*}\right\|_2^{2} \notag\\
		&= \left\|\bar{{\theta}}_{t}-{\theta}^{*}+\alpha_{t} \frac{\mathbf{u} ^\top\mathbf{v}}{n} \bar{{g}}(\bar{{\theta}}_{t}) -\alpha_{t} \left(\frac{\mathbf{u}^\top \mathbf{v}}{n} \bar{{g}}(\bar{{\theta}}_{t}) - \sum_{j=1}^{n} \frac{u_{j}}{n} y_{t}^{j}\right)\right\|_2^{2} \notag\\
		&\leq (1+\varepsilon)\left\|\bar{{\theta}}_{t}-{\theta}^{*}+\alpha_t\frac{\mathbf{u}^\top\mathbf{v}}{n}\bar{g}(\bar{\theta}_t)\right\|_2^{2}\notag\\
		&\quad + \left(1+\frac{1}{\varepsilon}\right) \alpha_{t}^{2} \left\|\frac{\mathbf{u}^\top \mathbf{v}}{n} \bar{{g}}(\bar{{\theta}}_{t}) - \sum_{j=1}^{n} \frac{u_{j}}{n} y_{t}^{j}\right\|_2^{2} \notag\\
		&\le (1+\varepsilon)\left(\left\|\bar{{\theta}}_{t}-{\theta}^{*}\right\|_2^{2} \right.\notag\\
		&\quad\left.- \left(2 (1-\gamma)\frac{\mathbf{u}^\top\mathbf{v}}{n}\alpha_{t} - 4  \left(\frac{\mathbf{u}^\top\mathbf{v}}{n}\right)^2\alpha_{t}^2\right) w\left\|\bar{{\theta}}_{t}-{\theta}^{*}\right\|_2^{2}\right)\notag\\
		&\quad+ \left(1+\frac{1}{\varepsilon}\right) \alpha_{t}^{2} \left\|\sum_{j=1}^{n} \frac{u_{j}}{n} (v_j{g}(\bar{{\theta}}_{t}) - y_{t}^{j})\right\|_2^{2} \notag\\
		&\le \left(1-\frac{(1-\gamma)w\mathbf{u}^\top\mathbf{v}}{2n}  \alpha_{t}\right) \left\|\bar{{\theta}}_{t}-{\theta}^{*}\right\|_2^{2} \notag\\
		&\quad+ \frac{2n}{(1-\gamma) w\mathbf{u}^\top\mathbf{v}} \alpha_{t} \left\|\sum_{j=1}^{n} \frac{u_{j}}{n} (v_j \bar{g}(\bar{{\theta}}_{t})-y_{t}^{j})\right\|_2^{2}, 
	\end{align}}
	where the first inequality follows from the fact that $(a+b)^2 \leq \left(1+\frac{1}{\epsilon}\right) b^2 + (1+\epsilon) a^2$, $\forall a, b, \epsilon > 0$, the second inequality follows from  \cite[Lemmas 1, 3, 4]{bhandari2018finite},
	and the third inequality follows from the fact that $\alpha_{t} \le \frac{(1-\gamma)n}{4\mathbf{u}^\top\mathbf{v}}$ and the setting $\varepsilon = \frac{(1-\gamma)w\mathbf{u}^\top\mathbf{v}\alpha_t}{2n-2(1-\gamma)w\mathbf{u}^\top\mathbf{v}\alpha_t}$.

	For the last term on the right-hand side of the above inequality,
	\begin{align}
		&\quad\frac{2n}{(1-\gamma)w\mathbf{u}^\top\mathbf{v} }\alpha_t\left\|\sum_{j=1}^{n} \frac{u_{j}}{n}(v_j {g}(\bar{{\theta}}_{t})-y_{t}^{j})\right\|_2^{2} \notag \\
		&= \frac{2n}{(1-\gamma)w\mathbf{u}^\top\mathbf{v} }\alpha_t\left\|\sum_{j=1}^{n} \frac{u_{j}}{n} (v_j \bar{g}(\bar{{\theta}}_{t}) - \frac{v_{j}}{n} \sum_{j=1}^n{g}_{j}({\theta}^j_{t})\right.\notag\\
		&\quad\left. + \frac{v_{j}}{n} \sum_{j=1}^n({g}^{j}({\theta}^j_{t})-y_{t}^{j})) 
		+ \left(v_j\bar{q}_t-{y}_{t}^j\right)\right\|_2^{2} \notag\\
		&\leq \frac{6n}{(1-\gamma)w \mathbf{u}^\top\mathbf{v}} \alpha_t \left( \left(\frac{\mathbf{u}^\top\mathbf{v}}{n}\right)^2 \left\| \bar{g}(\bar{\theta}_t)  - \frac{1}{n} \sum_{j=1}^n {g}^j(\theta_t^j) \right\|_2^2\right.\notag\\
		&\quad\left.+ \left(\frac{\mathbf{u}^\top\mathbf{v}}{n}\right)^2 \frac{1}{n}\sum_{j=1}^n\Vert g^j(\theta_t) - q_{t}^j \Vert_2^2+\sum_{j=1}^n\frac{u_j}{n} \Vert y_{t}^j - v_j \bar{y}_t \Vert_2^2\right)\notag\\
		&\leq \frac{6n}{(1-\gamma)w\mathbf{u}^\top\mathbf{v}}\alpha_t \left( \left(\frac{\mathbf{u}^\top\mathbf{v}}{n}\right)^2\frac{(1+\gamma)^2\bar{c}^2}{n}\|\Theta_t-\bar{\Theta}_t\|_{\widehat{\mathbf{W}}}\right.\notag\\
		&\quad\left.+\left(\frac{\mathbf{u}^\top\mathbf{v}}{n}\right)^2\frac{1}{n}\|\mathbf{e}_t\|^2+\bar{c}^2\|\mathbf{Y}_t-\bar{\mathbf{Y}}_t\|_{\widehat{\mathbf{M}}}\right), \notag
	\end{align}
	where the last inequality is obtained by the $(1+\gamma)$-Lipschitz
	continuity of $g(\cdot;\xi^i_{t+1})$ and Lemma \ref{lem99} ($ii$). Then
	\begin{align*}
		\Vert \bar{\theta}_{t+1} - \theta^* \Vert_2^2 &\leq \left( 1 - \frac{(1-\gamma)w\mathbf{u}^\top\mathbf{v}}{2n}\alpha_t \right) \Vert \bar{\theta}_t - \theta^* \Vert_2^2\\
		&\quad + \frac{6n}{(1- \gamma)w\mathbf{u}^\top\mathbf{v}}\alpha_t \left[ \left(\frac{\mathbf{u}^\top\mathbf{v}}{n}\right)^2 \frac{\bar{c}^2}{n} \Vert \Theta_t - \bar{\Theta}_t \Vert_{\widehat{\mathbf{W}}}^2 \right.\\&\left.\quad+ \left(\frac{\mathbf{u}^\top\mathbf{v}}{n}\right)^2 \frac{1}{n} \Vert \mathbf{e}_t \Vert^2 + \bar{c}^2 \Vert \mathbf{Y}_t - \bar{\mathbf{Y}}_t \Vert_{\widehat{\mathbf{M}}}^2 \right].
	\end{align*}
	Taking the expectation of both sides of the above inequality, we arrive at (\ref{e7}).
	The proof is complete.
\end{proof}

We are ready to establish the convergence rate of PP-DTD under the i.i.d. setting.
\begin{thm}\label{t1}
	Let
	{\small\begin{align*}
		\mathbf{V}_{t} &:= \frac{\rho_{\mathbf{W}}}{8C_{13}} \mathbb{E} [\Vert \mathbf{e}_{t} \Vert^2] + \frac{\rho_{\mathbf{W}}}{8C_{23}} \mathbb{E} [\Vert \mathbf{Y}_{t} - \bar{\mathbf{Y}}_{t} \Vert^2_{\widehat{\mathbf{M}}}]+ \mathbb{E} [\Vert \Theta_{t} - \bar{\Theta}_{t} \Vert_{\widehat{\mathbf{W}}}^2]\\
		&\quad  + \mathbb{E} [\Vert \bar{\theta}_{t} - \theta^* \Vert_2^2],
	\end{align*}
	}
	$\beta_t=\hat{c}\alpha_t$, $\hat{c} \ge  \frac{16 C_{41} C_{13}}{\rho_{\mathbf{W}}}$, $c'=\min \left\{ \frac{C_{44}}{2}, \hat{c}\right\}$, $c''=\left(\frac{3 \rho_{W}}{4 C_{13}}+\frac{9 \bar{c}^{2}}{8 C_{23}}\right)\sigma^2\hat{c}^2$ and 
	\begin{footnotesize}
		\begin{align*}
		\alpha_t &\le c_{min}\\
		&:=\min\left\{ \begin{aligned}
			& \frac{\hat{c} C_{23} \rho_{\mathbf{M}}}{2 C_{11} C_{23} \rho_{\mathbf{M}} + 2 C_{21} C_{13} \rho_{\mathbf{M}} + 16 C_{31} C_{13} C_{23}}, \\
			&\frac{4 C_{44} C_{13}}{\rho_{\mathbf{W}} (C_{14} + C_{24}) + 8 C_{23} C_{34}}, \frac{\rho_{\mathbf{W}}}{4(C_{33}+C_{43})}, \\
			& \frac{\rho_{\mathbf{M}}\rho_{\mathbf{W}} }{2\left(C_{22}\rho_{\mathbf{W}}+\frac{C_{12}}{C_{13}}C_{23}\rho_{\mathbf{W}}+8C_{32}C_{23}+8C_{42}C_{23}\right)},\\
			&\frac{\min \left\{ \frac{\rho_{\mathbf{M}}}{2}, \frac{\rho_{\mathbf{W}}}{2} \right\}}{\min \left\{ \frac{C_{44}}{2}, \hat{c} \right\}},\frac{\|\mathbf{W}-\mathbf{I}\|}{\sqrt{2}(1+\gamma)},\frac{(1-\gamma)n}{4\mathbf{u}^\top\mathbf{v}},\frac{1}{2\hat{c}}
		\end{aligned} \right\}
		\end{align*}
	\end{footnotesize}
	where elements of   matrix $\mathbf{C}:=[C_{ij}]$ are given by
\begin{equation*}
	\begin{bmatrix}
		C_{11} \\[1ex] C_{21} \\[1ex] C_{31} \\[1ex] C_{41}
	\end{bmatrix}
	=
	\begin{bmatrix}
		\bar{c}^2+16(1+\gamma)^2  \\[1ex]
		\frac{3\bar{c}^2}{\rho_{\mathbf{M}}} (\hat{c}^2 + 8) \\[1ex]
		\frac{3\bar{c}^2}{\rho_{\mathbf{M}}} \\[1ex]
		\frac{6}{(1 - \gamma )w \mathbf{u}^\top\mathbf{v}} \left(\frac{\mathbf{u}^\top\mathbf{v}}{n}\right)^2
	\end{bmatrix},\quad\quad\quad\quad
\end{equation*}

\begin{equation*}
	\begin{bmatrix}
		C_{12} \\[1ex] C_{22} \\[1ex] C_{32} \\[1ex] C_{42}
	\end{bmatrix}
	=
	\begin{bmatrix}
		 16(1+\gamma)^2\bar{c}^2\\[1ex]
		\frac{24c^4}{\rho_{\mathbf{M}}} (1+\gamma)^2 \\[1ex]
		\frac{3\bar{c}^4}{\rho_{\mathbf{M}}} \\[1ex]
		\frac{6n\bar{c}^2}{(1-\gamma)w\mathbf{u}^\top\mathbf{v}}
	\end{bmatrix},\quad\quad\quad\quad\quad\quad\quad
\end{equation*}

\begin{equation*}
	\begin{bmatrix}
		C_{13} \\[1ex] C_{23} \\[1ex] C_{33} \\[1ex] C_{43}
	\end{bmatrix}
	=
	\begin{bmatrix}
		4 (1+\gamma)^2 (1+8\Vert \mathbf{W}-\mathbf{I} \Vert^2) \bar{c}^2 \\[1ex]
		\frac{48\bar{c}^4}{\rho_{\mathbf{M}}} \Vert \mathbf{W}-\mathbf{I} \Vert^{2} (1+\gamma)^2 \\[1ex]
		\frac{2\bar{c}^4}{\rho_{\mathbf{W}}} (1+\gamma)^2 \\[1ex]
		\frac{6\bar{c}^2}{(1-\gamma)w\mathbf{u}^\top\mathbf{v}} \left(\frac{\mathbf{u}^\top\mathbf{v}}{n}\right)^2
	\end{bmatrix},
\end{equation*}

\begin{equation*}
	\begin{bmatrix}
		C_{14} \\[1ex] C_{24} \\[1ex] C_{34} \\[1ex] C_{44}
	\end{bmatrix}
	=
	\begin{bmatrix}
		4 (1+\gamma)^2 n (3\hat{c}+8(1+\gamma)^2n) \\[1ex]
		\frac{48\bar{c}^2}{\rho_{\mathbf{M}}} (1+\gamma)^4 n \\[1ex]
		\frac{2\bar{c}^2}{\rho_{\mathbf{W}}} (1+\gamma)^2 \\[1ex]
		\frac{(1-\gamma )w\mathbf{u}^\top\mathbf{v}}{2n}
	\end{bmatrix},
\end{equation*}
	$\sigma^2$ is defined in Lemma \ref{lem2},  $\bar{c}$, $\rho_{\mathbf{W}}$ and $\rho_{\mathbf{M}}$ are defined in Lemma \ref{lem99}. Suppose that Assumptions \ref{ass2}, \ref{ass4} hold. Then 
	\begin{itemize}
		\item[(i)] If $\alpha_t = \frac{c_0}{(t+t_0)^{c_1}}$, $\beta_t = \frac{\hat{c} c_0}{(t+t_0)^{c_1}}$ with $k\in(0.5,1]$, $t_0\ge1$, $\frac{c_0}{t_0^{c_1}}\le c_{min}$ and $c_0\ge\frac{2c_1}{c'}$,
		
		\begin{equation*} 
			\mathbf{V}_T \leq\max \left\{ \mathbf{V}_{0} t_0^{c_1}, \frac{2 c'' c_0}{c'} \right\} \frac{1}{(T+t_0)^{c_1}}
		\end{equation*}
		for any $T\ge0$.
		
		\item[(ii)] If $\alpha_t \equiv \alpha$ with $\alpha<\frac{1}{c'}$, $$\mathbf{V}_T \le e^{-c'\alpha T}\mathbf{V}_0+{c''}\alpha$$ for any $T\ge0$.
	\end{itemize}
\end{thm}
\begin{proof}
	By Lemmas \ref{lem2}-\ref{lem5} and the fact that $\beta_t=\hat{c}\alpha_t$,  
	{\footnotesize\begin{align*}
		&\mathbf{V}_{t+1}\\
		 &\leq \left( 1 - 2\hat{c} \alpha_t + C_{11} \alpha_t^2 + \frac{C_{21} C_{13}}{C_{23}} \alpha_t^2 + \frac{8C_{31} C_{13}}{\rho_{\mathbf{W}}} \alpha_t^2\right.\\
		 &\quad \left. + \frac{8C_{41} C_{13}}{\rho_{\mathbf{W}}} \alpha_t \right)\frac{\rho_{\mathbf{W}}}{8C_{13}} \mathbb{E}[\Vert \mathbf{e}_{t} \Vert^2] \\
		&\quad+ \left( 1 - \rho_{\mathbf{M}} + C_{22} \alpha_t^2 + \frac{C_{12} C_{23}}{C_{13}} \alpha_t^2 + \frac{8C_{32} C_{23}}{\rho_{\mathbf{W}}} \alpha_t^2\right.\\
		&\quad \left. + \frac{8C_{42} C_{23}}{\rho_{\mathbf{W}}} \alpha_t \right) \frac{\rho_{\mathbf{W}}}{8C_{23}} \mathbb{E}[\Vert \mathbf{Y}_{t} - \bar{\mathbf{Y}}_{t} \Vert_{\widehat{\mathbf{M}}}^2] \\
		&\quad+ \left( 1 - \rho_{\mathbf{W}} + C_{33} \alpha_t^2 + \frac{\rho_{\mathbf{W}}}{4} + C_{43} \alpha_t \right) \mathbb{E}[\Vert \Theta_t - \bar{\Theta}_t \Vert_{\widehat{\mathbf{W}}}^2]\\
		&\quad+\left( \frac{3\rho_{\mathbf{W}}}{4C_{13}}\sigma^2 \hat{c}^2  + \frac{9\bar{c}^2}{8C_{23}} \sigma^2\hat{c}^2  \right)\alpha_t^2 \\
		&\quad+ \left( 1 - C_{44} \alpha_t + \frac{\rho_{\mathbf{W}} C_{14}}{8C_{13}} \alpha_t^2 + \frac{\rho_{\mathbf{W}} C_{24}}{8C_{13}} \alpha_t^2 + C_{34} \alpha_t^2 \right) \mathbb{E}[\Vert \bar{\theta}_t - \theta^* \Vert_2^2], \notag
	\end{align*}
	}where $C_{ij}$ denotes the $(i,j)$-th entry of matrix $\mathbf{C}$.

	Notice that $\alpha_t\le c_{min}$, $c'=\min \left\{ \frac{C_{44}}{2}, \hat{c},  \right\}$, $c''=\left(\frac{3 \rho_{W}}{4 C_{13}}+\frac{9 \bar{c}^{2}}{8 C_{23}}\right) \sigma^{2} \hat{c}^{2}$, 
	{\small\begin{align}
		&\mathbf{V}_{t+1}\notag\\
		 &\leq \left(1-\hat{c} \alpha_{t}\right) \frac{\rho_{\mathbf{W}}}{8 C_{13}} \mathbb{E}\left[\|\mathbf{e}_{t}\|^{2}\right] + \left(1-\frac{\rho_{\mathbf{M}}}{2}\right)\frac{\rho_{\mathbf{M}}}{8C_{23}} \mathbb{E}\left[\|\mathbf{Y}_{t} - \bar{\mathbf{Y}}_{t}\|^{2}_{\widehat{\mathbf{M}}}\right]\notag\\&\quad+ \left(1-\frac{\rho_{\mathbf{W}}}{2}\right) \mathbb{E}\left[\|\Theta_{t} - \bar{\Theta}_{t}\|^{2}_{\widehat{\mathbf{W}}}\right] + \left(1-\frac{C_{44} }{2}\alpha_t\right) E\left[\|\bar{\theta}_{t} - \theta^{*}\|_2^{2}\right]\notag\\&\quad+ \left(\frac{3 \rho_{\mathbf{W}}}{4 C_{13}}+\frac{9 \bar{c}^{2}}{8 C_{23}}\right)\sigma^2\hat{c}^2   \alpha_{t}^{2},\notag\\
		&\leq \left(1 - \min\left\{\frac{C_{44}}{2}, \hat{c}\right\} \alpha_{t}\right) \mathbf{V}_{t} + \left(\frac{3 \rho_{W}}{4 C_{13}}+\frac{9 \bar{c}^{2}}{8 C_{23}}\right) \sigma^{2} \hat{c}^{2} \alpha_{t}^{2}\notag\\
		&=(1-c'\alpha_t)\mathbf{V}_t+c''\alpha_t^2.\label{edg}
	\end{align}}

	\textbf{Part ($i$)}. The proof is by induction on $k$. Denote $Q = \max \left\{ \mathbf{V}_{0} t_0^{c_1}, \frac{2 c'' c_0}{c'} \right\}$. Obviously, when $t=0$, $\mathbf{V}_0 \le \frac{Q}{t_0^{c_1}}$.
	
	Assume that the inequality holds for some $t \ge 0$, i.e., 
	\begin{equation*}
		\mathbf{V}_t \le \frac{Q}{(t+t_0)^{c_1}}.
	\end{equation*}
	
	By  (\ref{edg}) and substituting the inductive hypothesis, we have
	\begin{align*}
		V_{t+1} &\le \frac{Q}{(t+t_0)^{c_1}} - \frac{Q c' c_0 - c'' c_0^2}{(t+t_0)^{2c_1}}\\
		&\le\frac{Q}{(t+t_0)^{c_1}} - \frac{Q\frac{c'c_0}{2}}{(t+t_0)^{2c_1}}\\
		&\le\frac{Q}{(t+t_0)^{c_1}} - \frac{\frac{Qc_1}{(t+t_0)^{1-c_1}}}{(t+t_0)^{2c_1}}
		\le\frac{Q}{(t+t_0+1)^{c_1}},
	\end{align*}
	where the second inequality is obtained by the fact that $Q \ge \frac{2 c'' c_0}{c'}$, the third inequality follows form the facts that $t_0^{1-k} \ge \frac{2c_1}{c'' c_0}$ and $\frac{Q k}{t_0^{1-k}}\ge\frac{Q k}{(t+t_0)^{1-k}}$, and the last inequality is obtained by the properties of Taylor expansion, i.e., $\frac{1}{(t+t_0+1)^{c_1}} \ge \frac{1}{(t+t_0)^{c_1}} - \frac{k}{(t+t_0)^{c_1+1}}$. Then, for all $t \ge 0$,
	\begin{equation*}
		\mathbf{V}_t \leq  \max \left\{ \mathbf{V}_{0} t_0^{c_1}, \frac{2 c'' c_0}{c'} \right\} \frac{1}{(t+t_0)^{c_1}}.
	\end{equation*}
	
	\textbf{Part ($ii$)}. Unrolling recursion of (\ref{edg}) gives
	\begin{align}
		\mathbf{V}_{t+1} \leq \prod_{l=0}^{t} (1 - c' \alpha_{l}) \mathbf{V}_{0} + c'' \sum_{l=0}^{t} \prod_{k=l+1}^{t} (1 - \alpha_{k}) \alpha_{l}^{2}. \label{e8}
	\end{align}

	Substituting $\alpha_t\equiv\alpha$ into (\ref{e8}), we obtain
	\begin{align*}
		\mathbf{V}_{T} &\leq \prod_{l=0}^{T-1} \left(1 - c' \alpha\right) \mathbf{V}_{0} + c'' \sum_{l=0}^{T-1} \prod_{k=l+1}^{T-1} \left(1 - \alpha\right) \alpha^2\\
		&\le\left(1-{c'}\alpha\right)^{\top}\mathbf{V}_0+{c''}\alpha\left[1- \left(1-{c'}\alpha\right)^{\top}\right]\\
		&\le e^{-c'\alpha T}\mathbf{V}_0+{c''}\alpha.
	\end{align*}
	The proof is complete.
\end{proof}
Theorem~\ref{t1} shows that PP-DTD achieves linear convergence to a neighborhood of the optimum under constant step-sizes, and a convergence rate of $\mathcal{O}(T^{-1})$ under decaying step-sizes, when the samples are i.i.d. This result matches the convergence rate of the standard TD algorithm in single-agent RL \cite{bhandari2018finite}. To the best of our knowledge, PP-DTD is the first distributed algorithm for policy evaluation in MARL over directed networks that attains this desired convergence rate.

\subsection{Convergence analysis under Markovian setting}

In this subsection, we present the convergence rate of PP-DTD under the Markovian setting,  that is the tuples $\bigl(s_t, \{r_t^i\}_{i\in\mathcal{V}}, s_{t+1}'\bigr)$ with $s_{t+1}'=s_{t+1}$ are collected along a single trajectory of a Markov chain.
The  proof under the Markovian setting follows a similar line to that of Theorem \ref{t1}, with the difference lying in the analysis of $\mathbb{E}\left[\|\mathbf{e}_t\|^2\right]$ due to the bias introduced by Markovian sampling. We use the technical results in \cite{bhandari2018finite,doan2021finite} to quantify the bias induced by Markov setting.
\begin{lem}\label{lem6}
	Suppose that Assumption \ref{ass2} holds, 
	\begin{itemize}
		\item [(i)]$\|g(\theta_t^i;\xi_t^i)\|_2\leq C_g$,
		\item[(ii)]$\|g^i(\theta_t^i)\|_2\leq C_g$, $\|g^i(\bar{\theta}_t)\|_2\leq C_g$,
	\end{itemize} 
	where $C_g:= r_{\text{max}}+2\mathcal{R}$.  
\end{lem}
\begin{proof}
	The proof is similar to the proof of  \cite[Lemma 6]{bhandari2018finite}.
\end{proof}

Note that the Markov chain associated with $P_\pi$ (defined in (\ref{def:r&P})) is irreducible and aperiodic.  For any given small constant $\alpha$, there exist a constant $C_{mix}>0$
\cite{doan2021finite} such that
\begin{equation}\label{mixing}
	\begin{aligned} 
		\left\| \mathbb{E}\left[A(\xi_t) - {A} \mid s_0 \right] \right\| &\le \alpha, \quad \forall  t \ge \tau(\alpha) \\ 
		\left\| \mathbb{E}\left[b^i(\xi_t) - b^i \mid s_0\right]\right \| &\le \alpha, \quad \forall i\in \mathcal{V},\quad \forall t \ge \tau(\alpha), 
	\end{aligned}
\end{equation}
where $\tau{(\alpha)}=C_{mix}\text{log}\left(\frac{1}{\alpha}\right)$ is the mixing time,
\begin{align*}
	&A :=  \Phi^{\top} \text{diag}(\mathbf{d}_\pi) (\gamma P_\pi \Phi - \Phi) ,\quad b^i :=\sum_{s \in \mathcal{S}} d_\pi(s) r^i_\pi(s)\phi(s),\\
	&A(\xi_t):=\phi(s_t)\left(\gamma \phi(s_{t+1})^\top \theta - \phi(s_t)^\top\right),\quad b^i(\xi_t):=r_t^i \phi(s_t).
\end{align*}

We are ready to establish the convergence rate of PP-DTD under the Markovian setting. 
\begin{thm}\label{t2}
	Let
	{\small\begin{align*}
		\widetilde{\mathbf{V}}_{t} &:= \frac{\rho_{\mathbf{W}}}{8C_{3}'} \mathbb{E} [\Vert \mathbf{e}_{t} \Vert^2] + \frac{\rho_{\mathbf{W}}}{8C_{23}} \mathbb{E} [\Vert \mathbf{Y}_{t} - \bar{\mathbf{Y}}_{t} \Vert^2_{\widehat{\mathbf{M}}}] + \mathbb{E} [\Vert \Theta_{t} - \bar{\Theta}_{t} \Vert_{\widehat{\mathbf{W}}}^2]\\
		&\quad + \mathbb{E} [\Vert \bar{\theta}_{t} - \theta^* \Vert_2^2],
	\end{align*}}
	$\beta_t=\hat{c}'\alpha_t$, $\hat{c}' \ge  \frac{32 C_{41} C_{3}'}{\rho_{\mathbf{W}}}$, $c'_1=\min \left\{ \frac{C_{44}}{2}, \frac{\hat{c}'}{2}\right\}$, $c_1''=\frac{\rho_{\mathbf{W}}C_{4}'}{8 C_{3}'}+\frac{9 \bar{c}^{2}}{8 C_{23}}\sigma^2\hat{c}'^2$ and
	{\footnotesize\begin{align*}
	\alpha_t &\le c_{min}\\
	&:= \min \left\{ \begin{aligned}
		& \frac{\hat{c}' C_{23} \rho_{\mathbf{M}}}{4 C_{1}' C_{23} \rho_{\mathbf{M}} + 4 C_{21} C_{3}' \rho_{\mathbf{M}} + 32 C_{31} C_{3}' C_{23}},\\
		& \frac{4 C_{44} C_{3}'}{\rho_{\mathbf{W}}   C_{24} + 8 C_{23} C_{34}}, \frac{\rho_{\mathbf{W}}}{4(C_{33}+C_{43})}, \\
		& \frac{\rho_{\mathbf{M}}\rho_{\mathbf{W}} C_{3}'}{2(C_{3}'C_{22}\rho_{\mathbf{W}}+C_{2}'C_{23}\rho_{\mathbf{W}}+8C_{32}C_{23}C_{3}'+8C_{42}C_{23}C_{3}')},\\
		&\frac{\min \left\{ {\rho_{\mathbf{M}}}, {\rho_{\mathbf{W}}}\right\}}{\min \left\{ {C_{44}}, \hat{c}' \right\}},\frac{\|\mathbf{W}-\mathbf{I}\|}{\sqrt{2}(1+\gamma)},\frac{1-\gamma}{4},\frac{1}{2\hat{c}'}
	\end{aligned} \right\},
	\end{align*} }where
	{\small\begin{align*}
		C'_{1}&:=7+32(1+\gamma^2),\quad C'_{2}:=(32(1+\gamma^2)+6)\bar{c}^2,\\
		C'_{3}&:=\left(32(1+\gamma)^2+2\right)\left\|\mathbf{W}-\mathbf{I}\right\|^2\bar{c}^2,\\
		C'_{4}&:=32(1+\gamma)^2(r_{max}+2\mathcal{R})^2+14n(r_{max}+2\mathcal{R})^2+(\mathcal{R}+1)^2,
	\end{align*}}$C_{ij}$ denotes the $(i,j)$-th entry of matrix $\mathbf{C}$, $\sigma^2$ is defined Lemma \ref{lem2},  $\bar{c}$, $\rho_{\mathbf{W}}$ and $\rho_{\mathbf{M}}$ are defined in Lemma \ref{lem99}. Suppose that Assumptions \ref{ass2}, \ref{ass4} hold. Then 
	\begin{itemize}
		\item[(i)] If $\alpha_t = \frac{c'_0}{(t+t'_0)^{c_1}}$, $\beta_t = \frac{\hat{c}' c_0'}{(t+t_0')^{c_1}}$ with $c_1\in(0.5,1]$, $t_0'>1$, $\frac{c_0'}{t_0'^{c_1}}\le c_{min}'$ and $c_0'\ge\frac{2c_1}{c_1'}$,
		
		\begin{equation*} 
			\widetilde{\mathbf{V}}_T \leq\max\left\{\widetilde{\mathbf{V}}_{t'^*}(t'^*+t_0')^{c_1}, \frac{2c_1'' c_0'}{c_1'}\right\} \frac{1}{(T+t'_0)^{c_1}}
		\end{equation*}
		for any $T\ge t'^*:= \min \left\{ t \;\middle|\; t>C_{mix}\log(\frac{t+t'_{0}}{c'_{0}}) \right\}$.
		
		\item[(ii)] If $\alpha_t \equiv \alpha$ with $\alpha<\frac{1}{c_1'}$, $$\widetilde{\mathbf{V}}_T \le e^{-c_1'\alpha T}\mathbf{V}_0+{c_1''}\alpha$$ for any $T\ge t^*:= \min \left\{ t \;\middle|\; t \ge {C}_{mix} \log\left(\frac{1}{\alpha}\right) \right\}$.
	\end{itemize}
\end{thm}
\begin{proof}
	Denote  $\mathbf{B}(\xi_t)=[b^1(\xi_t^1),b^2(\xi_t^2),\cdots,b^n(\xi_t^n)]^\top$, $\mathbf{B}=[b^1,b^2,\cdots,b^n]^\top$.
	
	By the definition of $\mathbf{e}_t$ in (\ref{ea}), we have
	{\small\begin{align*}
		&\mathbb{E}\left[\|\mathbf{e}_{t+1}\|^2\right]\\ 
		=& (1-\beta_t)^2 \mathbb{E}\left[\|\mathbf{e}_t\|^2\right]\\
		& + \mathbb{E}\left[\left\| (1-\beta_t) \underbrace{(\mathbf{G}_{{t}}^\mathbb{E}-\mathbf{G}_{t+1}^{\mathbb{E}}+\mathbf{G}_{t+1} - \tilde{\mathbf{G}}_{{t+1}})}_{\mathbf{Z}_t}\right.\right.\\
		&\left. \left.+ \beta_t\underbrace{(\mathbf{G}_{{t+1}} - {\mathbf{G}}_{{t+1}}^\mathbb{E})}_{\mathbf{Z}_t'} \right\|^2\right] \notag \\
		&+ 2(1-\beta_t)^2\mathbb{E} \left[\left\langle \mathbf{e}_t, \mathbf{Z}_t \right\rangle\right] + 2\beta_t (1-\beta_t) \mathbb{E}\left[\left\langle \mathbf{e}_t, \mathbf{Z}_t' \right\rangle\right] \notag\\
		\le& (1-\beta_t)^2\mathbb{E}\left[ \|\mathbf{e}_t\|^2\right] + 2(1-\beta_t)^2\mathbb{E}\left[\|\mathbf{Z}_t\|^2\right]+2\beta_t^2\mathbb{E}\left[\|\mathbf{Z}_t'\|^2\right]\notag\\
		\quad&+ 2(1-\beta_t)^2 \mathbb{E}\left[\left|\left\langle \mathbf{e}_t, \left((\mathbf{W}-\mathbf{I}) ({\Theta}_t - \bar{{\Theta}}_t) + \alpha_t\mathbf{Y}_t\right)\right.\right.\right.\\
		&\left.\left.\left.\mathbb{E}[(A(\xi_{t+1})-A)^\top|\mathbf{e}_t,{\Theta}_t,\mathbf{Y}_t] \right\rangle\right|\right]+ 2(1-\beta_t)\beta_t \mathbb{E}\left[\left|\left\langle \mathbf{e}_t,\right.\right.\right.\\
		&\left.\left.\left. \mathbb{E}[\Theta_{t+1}(A(\xi_{t+1})-A)^\top+(\mathbf{B}(\xi_{t+1})-\mathbf{B})|\mathbf{e}_t,\Theta_{t+1}] \right\rangle\right|\right].\notag
	\end{align*}}
	
	For the second term on the right-hand side of the above inequality, 
	\begin{align}
		\mathbb{E}[\|\mathbf{Z}_t\|^2] &= \mathbb{E}\left[ \left\| \mathbf{G}_{{t}}^\mathbb{E} - \mathbf{G}_{{t+1}}^\mathbb{E} + \mathbf{G}_{{t+1}} - \widetilde{\mathbf{G}}_{{t+1}} \right\|^2 \right] \notag\\
		&\le 4(1+\gamma)^2 \mathbb{E}[\|{\Theta}_{t+1} - {\Theta}_t\|^2] \notag\\
		&= 4(1+\gamma)^2 \mathbb{E}\left[ \left\| (\mathbf{W}-\mathbf{I}) ({\Theta}_t - \bar{{\Theta}}_t) \right.\right.\notag\\
		&\left.\left.\quad+ \alpha_t(\mathbf{Y}_t - \bar{\mathbf{Y}}_t + \bar{\mathbf{Y}}_t -\bar{\mathbf{G}}_t+ \bar{\mathbf{G}}_t) \right\|^2 \right] \notag\\
		&\le 16(1+\gamma)^2 \left( \|\mathbf{W}-\mathbf{I}\|^2 \bar{c}^2\mathbb{E}\left[\|{\Theta}_t - \bar{{\Theta}}_t\|_{\widehat{\mathbf{W}}}^2\right] \right.\notag\\
		&\left.\quad+ \alpha_t^2 \bar{c}^2\mathbb{E}\left[\|\mathbf{Y}_t - \bar{\mathbf{Y}}_t\|^2_{\widehat{\mathbf{M}}}\right] \right. \notag\\
		&\quad \left. + \alpha_t^2 \mathbb{E}\left[\|\mathbf{e}_t\|^2\right] + n(r_{max}+2\mathcal{R})^2\alpha_t^2  \right),\notag
	\end{align}
	where the second inequality follows from Lemma \ref{lem99}.
	
	When $t>\tau(\alpha_t)$, by the property of mixing time (\ref{mixing}),
	{\footnotesize\begin{align*}
		&\mathbb{E}\left[\left|\left\langle \mathbf{e}_t, \left((\mathbf{W}-\mathbf{I}) ({\Theta}_t - \bar{{\Theta}}_t) + \alpha_t\mathbf{Y}_t\right)\mathbb{E}[(A(\xi_{t+1})-A)^\top|\mathbf{e}_t,{\Theta}_t, \mathbf{Y}_t] \right\rangle\right|\right]\\
		&\le\mathbb{E}\left[ \alpha_t\left\|\mathbf{e}_t\right\|  \left\|(\mathbf{W}-\mathbf{I}) ({\Theta}_t - \bar{{\Theta}}_t) + \alpha_t(\mathbf{Y}_t - \bar{\mathbf{Y}}_t + \bar{\mathbf{Y}}_t -\bar{\mathbf{G}}_t+ \bar{\mathbf{G}}_t)\right\|  \right]\\
		&\le\frac{\alpha_t}{2}\mathbb{E}\left[\|\mathbf{e}_t\|^2\right]+\alpha_t\mathbb{E}\left[\left\|(\mathbf{W}-\mathbf{I}) ({\Theta}_t - \bar{{\Theta}}_t)\right\|^2\right]\\
		&\quad+3\alpha_t^3\mathbb{E}\left[  \bar{c}^2\left[\|\mathbf{Y}_t - \bar{\mathbf{Y}}_t\|^2_{\widehat{\mathbf{M}}}\right]+\|\mathbf{e}_t\|^2+ n(r_{max}+2\mathcal{R})^2\right]
	\end{align*}}
	and
	{\footnotesize\begin{align*}
		&\mathbb{E}\left[\left|\left\langle \mathbf{e}_t, \mathbb{E}[\Theta_{t+1}(A(\xi_{t+1})-A)^\top+(\mathbf{B}(\xi_{t+1})-\mathbf{B})|\mathbf{e}_t,\Theta_{t+1}] \right\rangle\right|\right]\\
		\le&\frac{1}{2}\mathbb{E}\left[\|\mathbf{e}_t\|^2\right]+(\mathcal{R}^2+1)\alpha_t^2.
	\end{align*}}
	
	Then
	\begin{align}
		&\mathbb{E}\left[\|\mathbf{e}_{t+1}\|^2\right]\notag\\
		 &\le (1-\beta_t+\alpha_t+32(1+\gamma)^2\alpha_t^2+6\alpha_t^3)\mathbb{E} \left[\|\mathbf{e}_t\|^2\right]\notag\\
		&\quad+\left(32(1+\gamma)^2+2\alpha_t\right)\left\|\mathbf{W}-\mathbf{I}\right\|^2\bar{c}^2\mathbb{E}\left[\|{\Theta}_t - \bar{{\Theta}}_t\|_{\widehat{\mathbf{W}}}^2\right]\notag\\
		&\quad+\left(32(1+\gamma)^2\alpha_t^2+6\alpha_t^3\right)\bar{c}^2\mathbb{E}\left[\|\mathbf{Y}_t - \bar{\mathbf{Y}}_t\|^2_{\widehat{\mathbf{M}}}\right]\notag\\
		&\quad+32n(1+\gamma)^2(r_{max}+2\mathcal{R})^2\alpha_t^2+8n(r_{max}+2\mathcal{R})^2\beta_t^2\notag\\
		&\quad+6n(r_{max}+2\mathcal{R})^2\alpha_t^3+(\mathcal{R}+1)^2\alpha_t^2,\label{e1155}
	\end{align}
	where the inequality is obtained by Young's inequality and $\mathbb{E}[\|\mathbf{Z}'_t\|^2]\le4n(r_{max}+\mathcal{R})^2$.

	By (\ref{e1155}), Lemmas \ref{lem3}-\ref{lem5} and the fact that $\beta_t=\hat{c}'\alpha_t$,  
{\small	\begin{align*}
		&\widetilde{\mathbf{V}}_{t+1}\\
		 &\leq \left( 1 - {\hat{c}} \alpha_t + C_{1}' \alpha_t + \frac{C_{21} C_{3}'}{C_{23}} \alpha_t^2 + \frac{8C_{31} C_{3}'}{\rho_{\mathbf{W}}} \alpha_t^2 \right.\\
		 &\left.\quad+ \frac{8C_{41} C_{3}'}{\rho_{\mathbf{W}}} \alpha_t \right) \frac{\rho_{\mathbf{W}}}{8C_{3}'} \mathbb{E}[\Vert \mathbf{e}_{t} \Vert^2] \\
		&\quad+ \left( 1 - \rho_{\mathbf{M}} + C_{22} \alpha_t^2 + \frac{C_{2}' C_{23}}{C_{3}'} \alpha_t^2 + \frac{8C_{32} C_{23}}{\rho_{\mathbf{W}}} \alpha_t^2 \right.\\
		&\left.\quad+ \frac{8C_{42} C_{23}}{\rho_{\mathbf{W}}} \alpha_t \right) \frac{\rho_{\mathbf{W}}}{8C_{23}} \mathbb{E}[\Vert \mathbf{Y}_{t} - \bar{\mathbf{Y}}_{t} \Vert_{\widehat{\mathbf{M}}}^2] \\
		&\quad+ \left( 1 - \rho_{\mathbf{W}} + C_{33} \alpha_t^2 + \frac{\rho_{\mathbf{W}}}{4} + C_{43} \alpha_t \right) \mathbb{E}[\Vert \Theta_t - \bar{\Theta}_t \Vert_{\widehat{\mathbf{W}}}^2]\\
		&\quad+\left( \frac{\rho_{\mathbf{W}}C_{4}'}{8C_{3}'}  + \frac{9\bar{c}^2}{8C_{23}} \sigma^2\hat{c}'^2  \right)\alpha_t^2 \\
		&\quad+ \left( 1 - C_{44} \alpha_t  + \frac{\rho_{\mathbf{W}} C_{24}}{8C_{3}'} \alpha_t^2 + C_{34} \alpha_t^2 \right) \mathbb{E}[\Vert \bar{\theta}_t - \theta^* \Vert_2^2]. \notag
	\end{align*}}
	Notice that $\alpha_t\le c_{min}'$, $c_1'=\min \left\{ \frac{C_{44}}{2}, \frac{\hat{c}'}{2},  \right\}$, $c_1''=\frac{\rho_{\mathbf{W}}C_{4}'}{8 C_{3}'}+\frac{9 \bar{c}^{2}}{8 C_{23}}\sigma^2\hat{c}'^2$, 
	\begin{align}
		\widetilde{\mathbf{V}}_{t+1}
		&\leq \left(1-\frac{\hat{c}'}{2} \alpha_{t}\right) \frac{\rho_{\mathbf{W}}}{8 C_{3}'} \mathbb{E}\left[\|\mathbf{e}_{t}\|^{2}\right]\notag\\ 
		&\quad+ \left(1-\frac{\rho_{\mathbf{M}}}{2}\right)\frac{\rho_{\mathbf{M}}}{8C_{23}} \mathbb{E}\left[\|\mathbf{Y}_{t} - \bar{\mathbf{Y}}_{t}\|^{2}_{\widehat{\mathbf{M}}}\right]\notag\\
		&\quad+ \left(1-\frac{\rho_{\mathbf{W}}}{2}\right) \mathbb{E}\left[\|\Theta_{t} - \bar{\Theta}_{t}\|^{2}_{\widehat{\mathbf{W}}}\right]\notag\\
		&\quad + \left(1-\frac{C_{44} }{2}\alpha_t\right) E\left[\|\bar{\theta}_{t} - \theta^{*}\|_2^{2}\right]\notag\\
		&\quad+\left( \frac{\rho_{\mathbf{W}}C_{4}'}{8C_{3}'}  + \frac{9\bar{c}^2}{8C_{23}} \sigma^2\hat{c}'^2  \right)\alpha_t^2 ,\notag\\
		&\leq \left(1 - \min\left\{\frac{C_{44}}{2}, \frac{\hat{c}'}{2}\right\} \alpha_{t}\right) \widetilde{\mathbf{V}}_{t}\notag \\
		&\quad+ \left( \frac{\rho_{\mathbf{W}}C_{4}'}{8C_{3}'}  + \frac{9\bar{c}^2}{8C_{23}} \sigma^2\hat{c}'^2  \right)\alpha_t^2\notag\\
		&=(1-c_1'\alpha_t)\widetilde{\mathbf{V}}_t+c_1''\alpha_t^2.\label{edg2}
	\end{align}

	By a similar analysis of Theorem \ref{t1}, we arrive at ($i$) and ($ii$). The proof is complete.
\end{proof}
Similar to the i.i.d. setting, PP-DTD achieves linear convergence to a neighborhood of the optimum under constant step-sizes, and a convergence rate of $\mathcal{O}(T^{-1})$ under decaying step-sizes. Different from the i.i.d. setting,   the convergence rates under markov setting  depend on the mixing times $t^*$ and $t'^*$ rather than on the initial step directly.  Recently, \cite{lin2024finite} proposed a push-sum-type distributed TD algorithm, Push-SA, over directed communication networks, and established a finite-time bound that asymptotically converges to zero under Markovian sampling.  To the best of our knowledge, PP-DTD is the first distributed algorithm for  policy evaluation in MARL over directed graphs that attains a convergence rate of $\mathcal{O}(T^{-1})$ under the Markovian setting.

\section{Numerical experiments}\label{sec:numerical}
\begin{figure}[htb]
	\subfigure[$n=20$]{
		\includegraphics[width=3.4in]{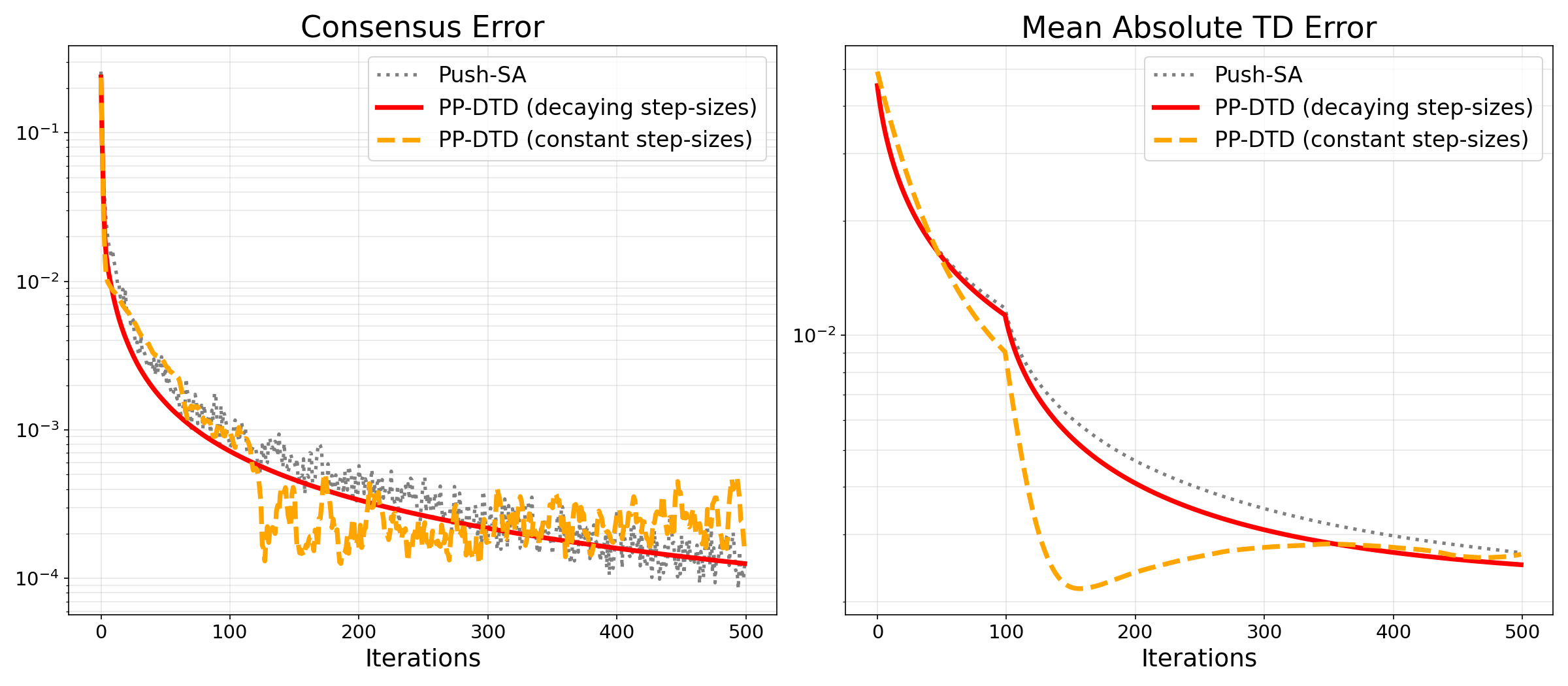}
	}\\
	\subfigure[$n=40$]{
		\includegraphics[width=3.4in]{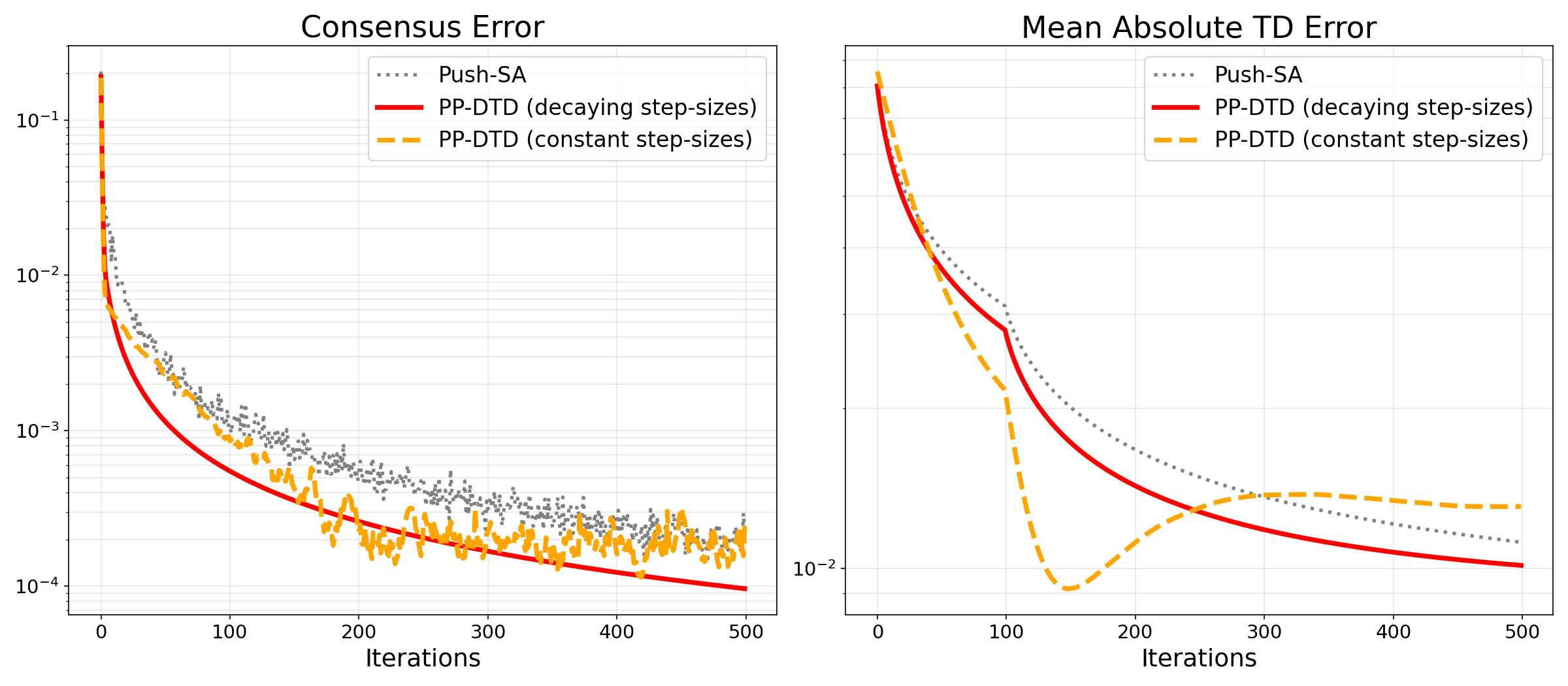}
	}\\
	\subfigure[$n=80$]{
		\includegraphics[width=3.4in]{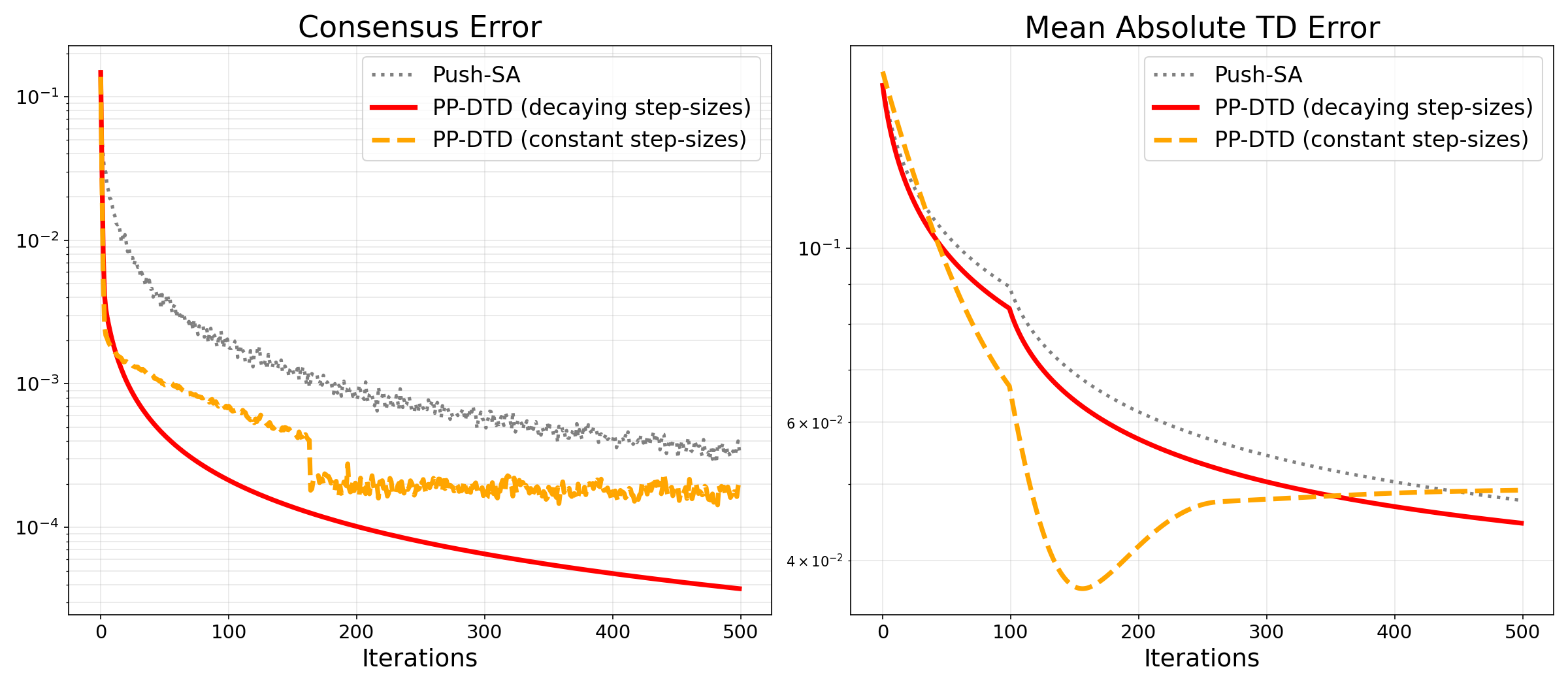}
	}
	\caption{i.i.d. setting.}\label{fig:iid}
\end{figure}

\begin{figure}[htb]
	\subfigure[$n=20$]{
		\includegraphics[width=3.4in]{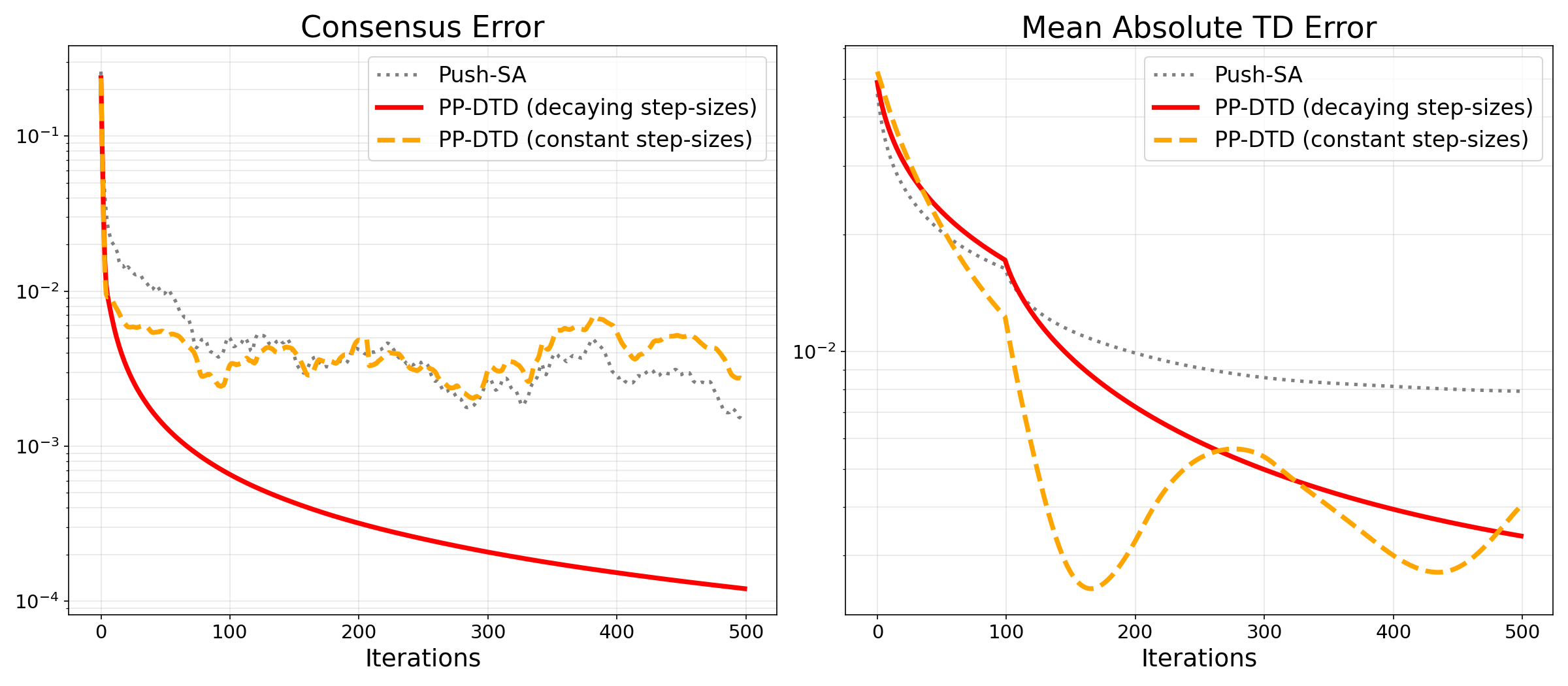}
	}\\
	\subfigure[$n=40$]{
		\includegraphics[width=3.4in]{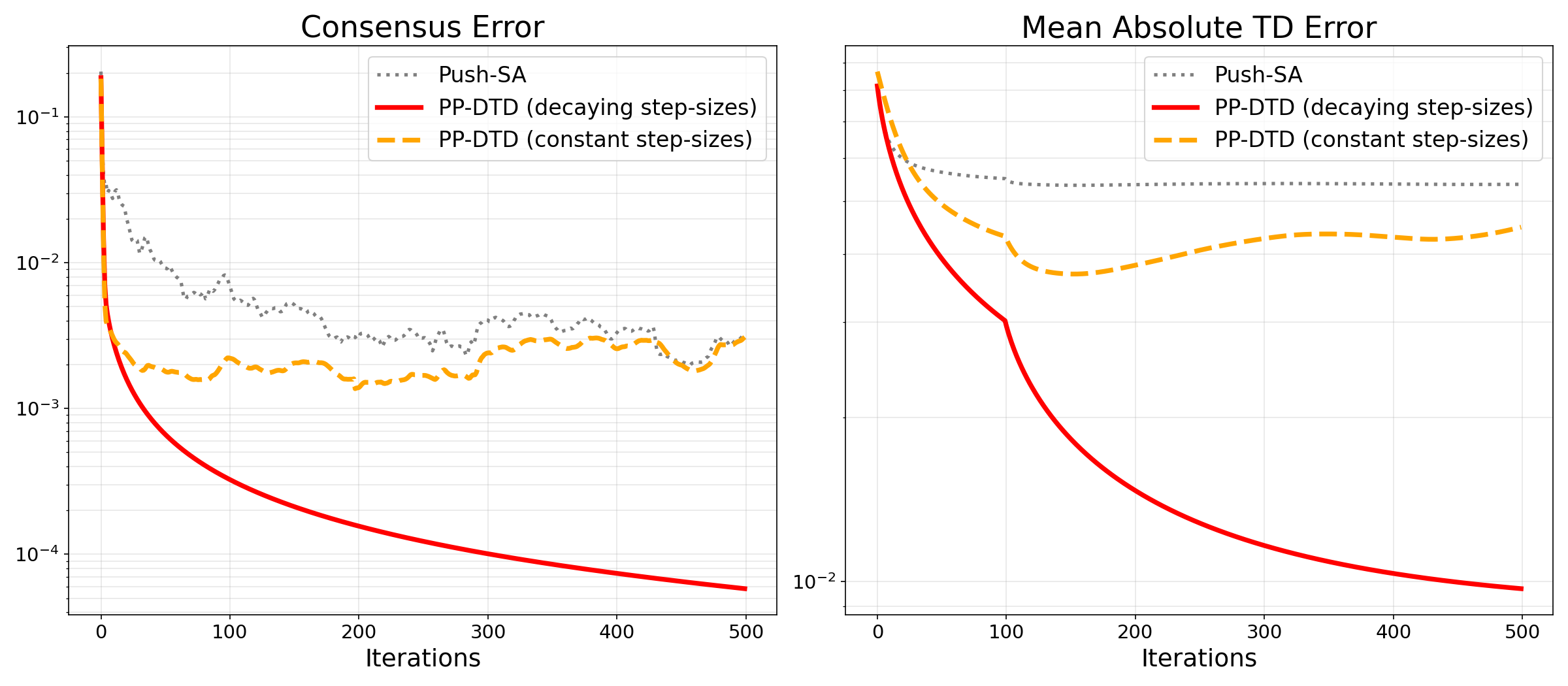}
	}\\
	\subfigure[$n=80$]{
		\includegraphics[width=3.4in]{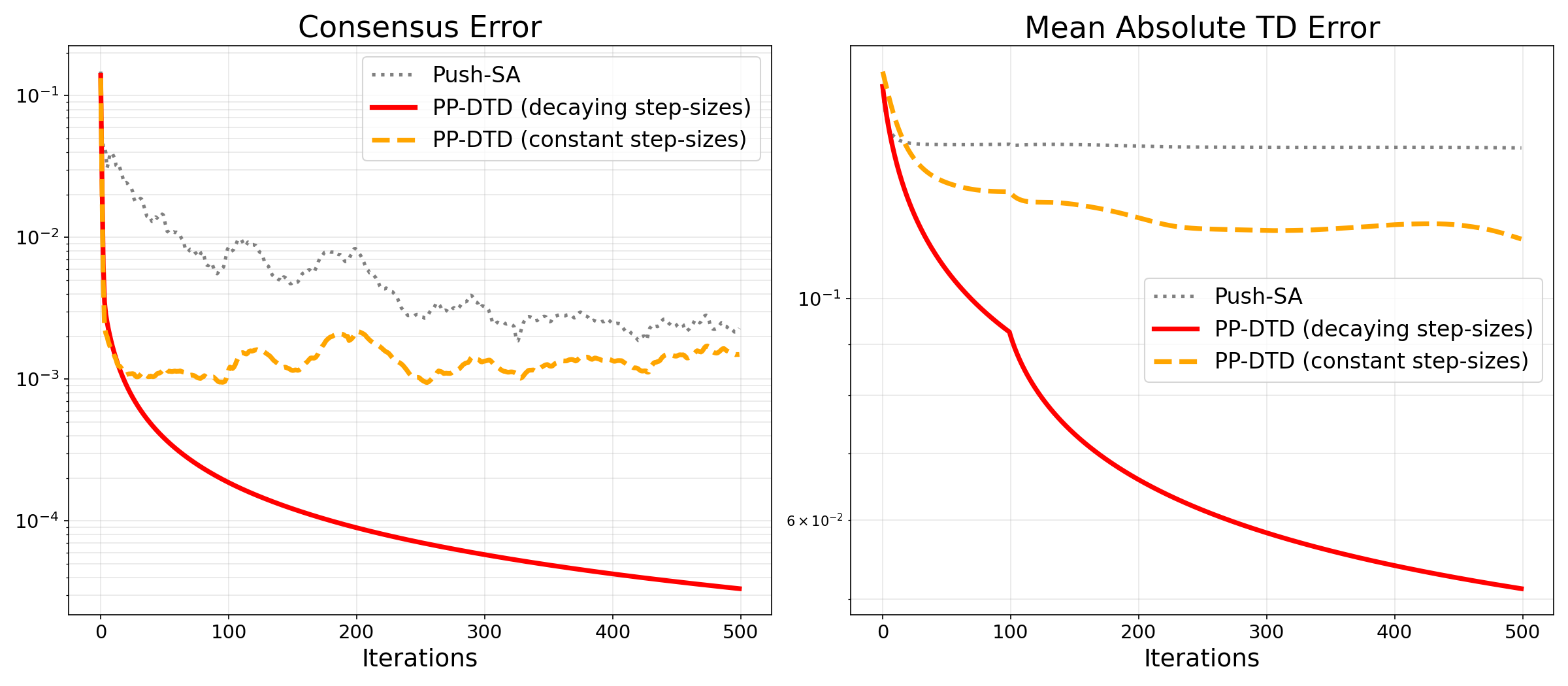}
	}
	\caption{Markovian setting.}\label{fig:markov}
\end{figure}
In this section, we evaluate the empirical performance of  PP-DTD on the cooperative navigation task \cite{lowe2017multi}, where $n$ agents are each given a personal target landmark and work together to cover all landmarks.  In every step, each agent follows its policy to select an action from the set $\{\text{up, down,  left, right, stay}\}$ and receives a local reward based on its distance to its assigned landmark, along with an extra penalty for collisions with other agents. The underlying communication graph $\mathcal{G}$ with $n$ agents is generated by adding random links to a ring network, where a directed link exists between any two nonadjacent nodes with probability $p=0.3$. The discount factor is $\gamma = 0.9$, and state features are constructed by radial basis functions.

We compare the performance of PP-DTD and the Push-SA method proposed in \cite{lin2024finite} across network sizes $n = 20, 40, 80$. For Push-SA, we set the step-sizes as $\alpha_t = \frac{a}{t+5}$. For PP-DTD, we set either $\alpha_t = \frac{a}{t+5}$ (decaying) or $\alpha_t \equiv a$ (constant), along with the momentum parameter $\beta_t = \frac{b}{t+5}$. The parameters $a$ and $b$ are optimized over $100$ equally spaced points in the interval $[10^{-3}, 5]$. For any agent $i$, we take
$\mathcal{G}_\mathbf{W}=\mathcal{G}_\mathbf{M}=\mathcal{G}$ and
\begin{equation*}
	\begin{aligned}
		&w_{ij}=\left\{
		\begin{aligned}
			&\frac{1}{|\mathcal{N}_{\mathbf{W},i}^{\text{in}}|+1},\quad j\in \mathcal{N}_{\mathbf{W},i}^{\text{in}},\\
			&1-\sum_{j\in\mathcal{N}_{\mathbf{W},i}^{\text{in}}}\mathbf{W}_{ij},\quad j=i,
		\end{aligned}\right.
		\\
		&m_{ji}=\left\{
		\begin{aligned}
			&\frac{1}{|\mathcal{N}_{\mathbf{M},i}^{\text{out}}|+1},\quad j\in\mathcal{N}_{\mathbf{M},i}^{\text{out}},\\
			&1-\sum_{j\in\mathcal{N}_{\mathbf{M},i}^{\text{out}}}\mathbf{M}_{ji},\quad j=i,
		\end{aligned}\right.
	\end{aligned}
\end{equation*}
where  $\mathcal{N}_{\mathbf{W},i}^{\text{in}}$ and $\mathcal{N}_{\mathbf{M},i}^{\text{out}}$  denote  the sets of in-neighbors and out-neighbors of agent $i$, 
$|\mathcal{N}_{\mathbf{W},i}^{\text{in}}|$ and $|\mathcal{N}_{\mathbf{M},i}^{\text{out}}|$ are the  cardinality of $\mathcal{N}_{\mathbf{W},i}^{\text{in}}$ and $\mathcal{N}_{\mathbf{M},i}^{\text{out}}$.

We run PP-DTD and Push-SA for 10 times, and record the averaged consensus error among agents and the averaged mean absolute TD error. Figure \ref{fig:iid} presents the results of i.i.d. setting. We can observe from Figure 1 that
PP‑DTD with decaying step‑sizes attains the lowest consensus and value estimation errors. This advantage may arise from two design features: a push‑pull structure that corrects weight imbalance in directed communication graphs, and a hybrid variance reduction technique that reduces variance. On the other hand, PP‑DTD with constant step‑sizes achieves the fastest decrease initially and then fluctuates in the final phase. Therefore, the constant‑step‑size variant is well suited for applications that need rapid convergence and can tolerate a modest final error.   
Similarly, Figure \ref{fig:markov} record the averaged consensus error among agents and the averaged absolute TD error, under the Markovian setting.
In Figure \ref{fig:markov}, both PP‑DTD (with projection radius $\mathcal{R}=5$) and Push‑SA achieve comparable performance to  the i.i.d. setting. Moreover, PP‑DTD exhibits more stable performance across different network sizes and attains significantly lower consensus and TD errors. This stable behavior and better final error are attributed to the projection step, which controls the bias introduced by Markovian sampling.

\section{Conclusion}
\label{sec:conclusion}

In this work, we propose a TD learning based algorithm, named PP-DTD, for the policy evaluation problem in MARL over directed communication networks. We show that PP-DTD achieves a non-asymptotic convergence rate of $\mathcal{O}({T}^{-1/2})$ under constant step-sizes and $\mathcal{O}({T^{-1}})$ under decaying step-sizes if the sample is independent and identically distributed  or  Markovian. Experiments on cooperative navigation tasks confirm that PP-DTD demonstrates favorable convergence performance and scalability and robustness. Future work will explore uncertainty quantification for PP-DTD, extend PP-DTD to general TD($\lambda$) learning,  and develop efficient policy gradient-based algorithms for MARL over directed communication graphs. 

\section*{Acknowledgment}

 The research is supported by the NSFC \#12471283, the NSFC \#12401418,  and  Fundamental Research Funds for the Central Universities DUT24LK001.

\footnotesize
\bibliographystyle{IEEEtran}
\bibliography{mybib}

\end{document}